\documentclass[12pt]{article}

\usepackage{setspace}

\usepackage{amsmath,amssymb,amsthm}
\usepackage{epsfig}
\usepackage[latin1]{inputenc}
\usepackage{graphicx} 
\usepackage{epstopdf} 
\begin{document}

\newtheorem{theorem}{Theorem}
\newtheorem{proposition}[theorem]{Proposition}
\newtheorem{remark}[theorem]{Remark}
\newtheorem{lemma}[theorem]{Lemma}
\newtheorem{corollary}[theorem]{Corollary}
\newtheorem{definition}[theorem]{Definition}
\newcommand{\RR}{\mathbb{R}}
\newcommand{\NN}{\mathbb{N}}
\newcommand{\ZZ}{\mathbb{Z}}
\newcommand{\PP}{\mathbb{P}}
\newcommand{\EE}{\mathbb{E}}
\newcommand{\Var}{\mathbb{V}{\rm ar}}
\newcommand{\Cov}{\mathbb{C}{\rm ov}}
\newcommand{\Id}{\text{Id}}
\newcommand{\dive}{\operatorname{div}}
\newcommand{\eps}{\varepsilon}

\newcommand{\dps}{\displaystyle}
\newcommand{\dis}{\displaystyle}

\newcommand{\D}{{\cal{D}}}

\def\longrightharpoonup{\relbar\joinrel\rightharpoonup}

\title{Variance reduction using antithetic variables for a nonlinear
  convex stochastic homogenization problem}
\author{F. Legoll$^{1,3}$ and W. Minvielle$^{2,3}$\\
{\footnotesize $^1$ Laboratoire Navier, \'Ecole Nationale des Ponts et
Chauss\'ees, Universit\'e Paris-Est,}\\
{\footnotesize 6 et 8 avenue Blaise Pascal, 77455 Marne-La-Vall\'ee
Cedex 2, France}\\
{\footnotesize \tt legoll@lami.enpc.fr}\\
{\footnotesize $^2$ CERMICS, \'Ecole Nationale des Ponts et
Chauss\'ees, Universit\'e Paris-Est,}\\
{\footnotesize 6 et 8 avenue Blaise Pascal, 77455 Marne-La-Vall\'ee
Cedex 2, France}\\
{\footnotesize \tt william.minvielle@cermics.enpc.fr}\\
{\footnotesize $^3$ INRIA Rocquencourt, MICMAC team-project,}\\
{\footnotesize Domaine de Voluceau, B.P. 105,
78153 Le Chesnay Cedex, France}
}
\date{\today}

\maketitle

\vspace{-4mm}

\begin{abstract}
We consider a nonlinear convex stochastic homogenization problem, in a
stationary setting. In
practice, the deterministic homogenized energy density can only be
approximated by a random apparent energy density, obtained by solving the
corrector problem on a truncated domain. 

We show that the technique of antithetic variables can be used to reduce
the variance of the computed quantities, and thereby decrease the
computational cost at equal accuracy. This leads to an efficient
approach for approximating expectations of the apparent homogenized
energy density and of related quantities. 

The efficiency of the approach is numerically illustrated on several
test cases. Some elements of analysis are also provided. 
\end{abstract}

\vspace{1mm}



\section{Introduction}

In this article, we consider some theoretical and numerical questions
related to variance reduction techniques for some nonlinear convex
stochastic homogenization problems. In short, we show here
that a technique based on antithetic variables can be used in that
context, provide some elements of analysis, and demonstrate numerically
the efficiency of that approach on several test cases. 
This work is a follow-up of the
articles~\cite{banff,mprf,cedya} where the same questions are considered
for a {\em linear} elliptic equation in divergence form. 

The stochastic homogenization problem we consider here writes as follows. 
Let $\D$ be an open bounded domain of $\RR^d$ and $2 \leq p < \infty$. 
We consider the highly oscillatory problem
\begin{equation}
\label{eq:pb0-stoc}
\inf \left\{ \int_{\D} W\left(\frac{x}{\eps}, \omega, \nabla u(x) \right)
  \, dx 
- \int_{\D} f(x) u(x) dx, \quad
u \in W^{1,p}_0(\D) \right\}
\end{equation}
for some $f$ and some random smooth field $W$, which is stationary in a
sense made precise below, and satisfies some convexity and growth
conditions such that, for any $\eps > 0$, problem~\eqref{eq:pb0-stoc} is
well-posed. See Section~\ref{sec:setting}
below for a precise description of the mathematical setting, which has
been introduced in~\cite{DalMaso-Modica,DalMaso-Modica2}. 
A classical example that motivated this framework is when
$$
W(y,\omega,\xi) = \frac{1}{p} a\left(y,\omega\right) \left| \xi
\right|^p,
$$
where $a$ is stationary (see e.g.~\cite[page 382]{DalMaso-Modica}).

In~\eqref{eq:pb0-stoc}, $\varepsilon$ denotes a supposedly small,
positive constant that models 
the smallest possible scale present in the problem. For $\varepsilon$
small, it is extremely expensive, in practice, to directly
attack~(\ref{eq:pb0-stoc}) with a numerical discretization. A useful
practical approach is to {\em first} approximate~\eqref{eq:pb0-stoc} by
its associated homogenized problem, which reads
\begin{equation}
\label{eq:pb0-stoc-homog}
\inf \left\{ \int_{\D} W^\star\left(\nabla u(x) \right) \, dx -
 \int_{\D} f(x) u(x) dx, \quad u \in W^{1,p}_0(\D) \right\},
\end{equation}
and \emph{next} numerically solve the latter problem. The two-fold advantage
of~\eqref{eq:pb0-stoc-homog} as compared to~\eqref{eq:pb0-stoc} is that
\emph{it is deterministic} and \emph{it does not involve the small scale
 $\varepsilon$}. 

This simplification comes at a price. The homogenized
energy density $W^\star$ in~\eqref{eq:pb0-stoc-homog} is given by
an integral involving a so-called corrector function, solution to a
nonlinear problem
(see~\eqref{eq:champ-homog} below for a precise formula). 
As most often in stochastic homogenization, this corrector problem
is set on the {\em entire} space $\RR^d$. In practice,
approximations are therefore in order. A standard approach (see
e.g.~\cite{bourgeat} in the linear setting) is to generate realizations
of the energy density $W$
over a finite, supposedly large volume at the
microscale, that we denote $Q_N$, and
approach the homogenized energy density by some empirical means using approximate
correctors computed on $Q_N$. Although the
\emph{exact} homogenized density $W^\star$ is deterministic, its practical
approximation is random, due to the truncation procedure. It is then
natural to generate several realizations. However,
efficiently averaging over these realizations require to understand how
variance affects the result. This is the purpose of the present article
to investigate some questions in this direction, both from the
theoretical and numerical standpoints. 

Before proceeding and for the sake of consistency, we
now present the framework of nonlinear stochastic
homogenization we adopt, and situate the questions under consideration
in a more general context.

\subsection{Homogenization theoretical setting}
 \label{sec:setting}

To begin with, we introduce the basic
setting of stochastic homogenization we will employ. We refer
to~\cite{engquist-souganidis} for a general, numerically oriented 
presentation, and to~\cite{blp,cd,jikov} for classical textbooks. 
We also refer to~\cite{enumath} and the review article~\cite{singapour}
(and the extensive bibliography contained therein) for a 
presentation of our particular setting. 
Throughout this article, $(\Omega, {\mathcal F}, \PP)$ is a
probability space and we denote by $\dis \EE(X) =
\int_\Omega X(\omega) d\PP(\omega)$ the expectation value of any
random variable $X\in L^1(\Omega, d\PP)$. We next fix
$d\in {\mathbb N}^\star$ (the ambient physical dimension), and assume that
the group $(\ZZ^d, +)$ acts on 
$\Omega$. We denote by $(\tau_k)_{k\in \ZZ^d}$ this action, and
assume that it preserves the measure $\PP$, that is, for all
$\displaystyle k\in \ZZ^d$ and all $A \in {\cal F}$, $\displaystyle
\PP(\tau_k A) = \PP(A)$. We assume that the action $\tau$ is {\em
  ergodic}, that is,
if $A \in {\mathcal F}$ is such that $\tau_k A = A$ for any $k \in
\ZZ^d$, then $\PP(A) = 0$ or 1. 
In addition, we define the following notion of stationarity
(see~\cite[Section 2.2]{singapour}): a function 
$F \in L^1_{\rm loc}\left(\RR^d, L^1(\Omega)\right)$ is said to be
{\em stationary} if, for all $k \in \ZZ^d$,
\begin{equation}
\label{eq:stationnarite-disc}
F(y+k, \omega) = F(y,\tau_k\omega)
\ \
\text{almost everywhere and almost surely.}
\end{equation}
In this setting, the ergodic theorem~\cite{krengel,shiryaev,tempelman}
can be stated as follows: 
{\it Let $F\in L^\infty\left(\RR^d, L^1(\Omega)\right)$ be a stationary random
variable in the above sense. For $k = (k_1,k_2, \dots, k_d) \in \ZZ^d$,
we set $\displaystyle |k|_\infty = \sup_{1\leq i \leq d} |k_i|$. Then
$$
\frac{1}{(2N+1)^d} \sum_{|k|_\infty\leq N} F(y,\tau_k\omega)
\mathop{\longrightarrow}_{N\rightarrow \infty}
\EE\left(F(y,\cdot)\right) \ \ \mbox{in $L^\infty(\RR^d)$, almost surely}.
$$
This implies that (denoting by $Q$ the unit cube in $\RR^d$)
$$
F\left(\frac x \varepsilon ,\omega \right)
\mathop{\longrightharpoonup}_{\varepsilon \rightarrow 0}^*
\EE\left(\int_Q F(y,\cdot)dy\right) \quad \mbox{in $L^\infty(\RR^d)$, almost surely}.
$$
}

The purpose of the above setting is simply to
formalize that, even though realizations may vary, the function $F$ at
point $y \in \RR^d$ and the function $F$ at point $y+k$, $k \in
\ZZ^d$, share the same law. In the homogenization context we now turn
to, this means that the local, microscopic environment (encoded in the
energy density $W$) is everywhere the same
\emph{on average}. From this, homogenized, macroscopic properties will follow.

\medskip

We now describe more precisely the multiscale random problem~\eqref{eq:pb0-stoc}. The
domain ${\mathcal D}$ is a regular (in the sense its boundaries are
Lipschitz-continuous) bounded domain of ${\mathbb R}^d$. 
The right-hand side function $f$ belongs to $L^{p'}(\D)$, with
$1/p+1/p'=1$ (hence $f$ 
is indeed in the dual space of $L^p(\D)$). 
For any $\xi \in \RR^d$, the random field $y,\omega \mapsto
W(y,\omega,\xi)$ is assumed stationary in the
sense~\eqref{eq:stationnarite-disc}. We assume that it is continuous
(and even $C^3$) with respect to the $\xi$ variable, and that it is
measurable with respect to the $y$ argument.
We also assume that there exists $c_2
\geq c_1 > 0$ such that
\begin{equation}
\label{eq:Caratheodory}
\forall y \in \RR^d, \ \ \forall \omega \in \Omega, \ \ \forall \xi \in
\RR^d, 
\quad c_1 |\xi|^p \leq W(y,\omega, \xi) \leq c_2(1+|\xi|^p).
\end{equation}
Furthermore, we assume henceforth that $W$ is
{\em strictly convex} with respect to the argument $\xi$, in the sense
that 
\begin{equation}
\label{eq:hyp_strict_convex}
\forall \eta \in \RR^d, \quad \forall \xi \in \RR^d, \quad
\eta^T \partial^2_\xi W(y,\omega,\xi) \eta > 0 
\ \text{ a.e. and a.s.},
\end{equation}
where $\partial^2_\xi W \in \RR^{d \times d}$ is the Hessian matrix of
$\xi \mapsto W(y,\omega,\xi)$. A more demanding assumption is that $W$ is
$\alpha$-{\em convex} with respect to the argument $\xi$, in the sense
that there exists $\alpha > 0$ such that
\begin{equation}
\label{eq:hyp_alpha_convex}
\forall \eta \in \RR^d, \quad \forall \xi \in \RR^d, \quad
\eta^T \partial^2_\xi W(y,\omega,\xi) \eta \geq \alpha \, \eta^T \eta
\ \text{ a.e. and a.s.}
\end{equation}
Unless otherwise stated, we only assume~\eqref{eq:hyp_strict_convex} in
the sequel. When needed, we will explicitly
assume~\eqref{eq:hyp_alpha_convex}.

Under~\eqref{eq:Caratheodory} and~\eqref{eq:hyp_strict_convex}, the
variational problem~\eqref{eq:pb0-stoc} is 
well-posed. In addition, the homogenized limit of~\eqref{eq:pb0-stoc}
has been identified in~\cite{DalMaso-Modica,DalMaso-Modica2} (see
also~\cite[Theorem 3.1]{neukamm}): the unique solution
$u^\eps(\cdot,\omega)$ to~\eqref{eq:pb0-stoc} converges (weakly in
$W^{1,p}(\D)$ and strongly in $L^p(\D)$, almost surely)
to some deterministic function $u^\star \in
W^{1,p}(\D)$, solution to~\eqref{eq:pb0-stoc-homog}, where the
homogenized energy density $W^\star$ is given, for any $\xi \in \RR^d$,
by 
\begin{equation}
\label{eq:champ-homog}
W^\star(\xi) = \lim_{N \to \infty} 
\inf \left\{ \frac{1}{|Q_N|} 
\int_{Q_N} W\left(y,\omega, \xi + \nabla w(y) \right) \, dy,
\quad w \in W^{1,p}_\#(Q_N) \right\}
\end{equation}
where $Q_N = (-N,N)^d$ and where $W^{1,p}_\# (Q_N)$ denotes the set of
functions that belong to $W^{1,p}_{\rm loc}(\RR^d)$ and are
$Q_N$-periodic. The convergence in~\eqref{eq:champ-homog} holds almost surely.

\subsection{The questions we consider}
\label{sec:questions}

In practice, we cannot compute $W^\star(\xi)$, and have to restrict
ourselves to finite size domains. We therefore introduce
\begin{equation}
\label{eq:correc-random_N}
W^\star_N(\omega,\xi) := 
\inf \left\{ \frac{1}{|Q_N|} \int_{Q_N} W\left(y,\omega, \xi + \nabla
    w(y) \right) \, dy, 
\quad w \in W^{1,p}_\#(Q_N) \right\}
\end{equation}
and readily see from~\eqref{eq:champ-homog} that 
$$
W^\star(\xi) = \lim_{N \to \infty} W^\star_N(\omega,\xi) \text{ a.s.}
$$
As briefly explained above, although $W^\star$ itself is a deterministic
object, its practical approximation $W^\star_N$ is random. It is only in
the limit of infinitely large domains $Q_N$ that the deterministic 
value is attained. This is a standard situation in stochastic
homogenization. 

Many studies have been recently devoted (at least in the linear case) to
establishing sharp estimates on the convergence of the random apparent
homogenized quantities (computed on $Q_N$) to the exact deterministic
homogenized quantities. We refer e.g. to~\cite{bourgeat,gloria-otto} and
to the comprehensive discussion of~\cite[Section 1.2]{mprf}. 
We take here the problem from a slightly different perspective. We
observe that the error
$$
W^\star(\xi) - W_N^\star(\omega,\xi)
=
\Big(
W^\star(\xi) - \EE \left[ W_N^\star(\cdot,\xi) \right]
\Big)
+
\Big(
\EE \left[ W_N^\star(\cdot,\xi) \right]
- W_N^\star(\omega,\xi)
\Big)
$$
is the sum of a systematic error (the first term in the above right-hand
side) and of a statistical error (the second term in the above right-hand
side).
We {\em focus here on the statistical error}, and propose approaches to
reduce the confidence interval of empirical means approximating
$\EE \left[ W_N^\star(\cdot,\xi) \right]$ (or similar quantities),
for a given truncated domain $Q_N$.

\medskip

Recall that a standard technique to compute an approximation of $\EE \left[
  W^\star_N(\cdot, \xi) \right]$ is to consider several independent and
identically distributed realizations of the energy density $W$, solve for each of
them the corrector problem~\eqref{eq:correc-random_N} (thereby obtaining
several i.i.d. values
$W^{\star,m}_N(\omega,\xi)$), and proceed following a Monte Carlo approach:
$$
\EE \left[ W^\star_N(\cdot, \xi) \right]
\approx
I_{2M} := \frac{1}{2M} \sum_{m=1}^{2M} W^{\star,m}_N(\omega,\xi).
$$
In view of the Central Limit Theorem, we know that our quantity of
interest $\EE \left[ W^\star_N(\cdot, \xi) \right]$ lies in the
confidence interval
$$
\left[
I_{2M} - 1.96 \frac{\sqrt{\Var \left[ W^\star_N(\cdot, \xi)\right]}}{\sqrt{2M}}
,
I_{2M} + 1.96 \frac{\sqrt{\Var \left[ W^\star_N(\cdot, \xi)\right]}}{\sqrt{2M}}
\right]
$$
with a probability equal to 95 \%. 

In this article, we show that, using a well known variance reduction
technique, the technique of {\em antithetic variables}~\cite[page
27]{liu}, we can design a practical approach that, for
finite $N$ and any vector $\xi$, allows to compute a better
approximation of $\EE \left[ W^\star_N(\cdot, \xi) \right]$ (and
likewise for similar homogenized quantities). Otherwise stated, for an
equal computational cost, the approach provides a more accurate
(i.e. with a smaller confidence interval) approximation. We thereby
extend to this nonlinear convex setting the results of~\cite{banff,mprf,cedya}
obtained in the linear case. 

\medskip

Our article is articulated as follows. In Section~\ref{sec:main_results}, we
describe the proposed approach, and state our main results. The
ingredients to prove these results are collected in
Sections~\ref{sec:antithetic},~\ref{sec:expression}
and~\ref{sec:monoton}. The actual proof of our main results is performed
in Section~\ref{sec:proof}. We make there several structural assumptions on
the form of the energy density $W$ to obtain these variance reduction
results. In Section~\ref{sec:examples}, we describe a general class of
examples for which our assumptions are indeed satisfied. We next turn in
Section~\ref{sec:num} to some illustrative numerical examples, where we
demonstrate the efficiency of the approach, even in cases where the
theoretical analysis is incomplete. 

\section{Description of the proposed approach and main results}
\label{sec:math}

\subsection{Statement of our main results}
\label{sec:main_results}

This section is devoted to the presentation and the analysis of our
approach. We first focus on estimating the expectation 
$\EE\left[W^\star_N(\cdot,\xi)\right]$ of the apparent homogenized
energy density (see Section~\ref{sec:0}). Our variance reduction result,
Proposition~\ref{prop:strong}, shows that the technique of antithetic
variables is indeed efficient. As often the case, it is difficult to
{\em quantitatively} assess how efficient the approach is, and this will
be the purpose of the numerical tests described in Section~\ref{sec:num}
to address this question. 

We then turn to the estimation of the first (and next second)
derivatives of $W^\star_N(\cdot,\xi)$ with respect to $\xi$. These
quantities naturally appear when one solves the convex homogenized
problem~\eqref{eq:pb0-stoc-homog} (approximating $W^*$ by $W^\star_N(\omega,\cdot)$),
e.g. using a Newton algorithm. For these two quantities, our result is
restricted to the one-dimensional setting. See Section~\ref{sec:1} and
Proposition~\ref{prop:var-red-der} for the first derivative, and
Section~\ref{sec:2} and Proposition~\ref{prop:var-red-der2} for the
second derivative.

\medskip

Sections~\ref{sec:antithetic},~\ref{sec:expression},~\ref{sec:monoton}
and~\ref{sec:proof} are devoted to the proof of the results stated
here. In Section~\ref{sec:examples}, we discuss an explicit class of
energy densities $W$ that falls into our framework.

\subsubsection{Variance reduction on the homogenized energy density}
\label{sec:0}

In this section, we make the following two \emph{structure assumptions}
on the rapidly oscillating field $W$ of~\eqref{eq:pb0-stoc}. First, we
assume that, for any $N$, there exists an integer $n$ (possibly $n =
|Q_N|$, but not necessarily) and a function ${\cal A}$, defined on $Q_N
\times \RR^n \times \RR^d$, such that the field $W(y,\omega,\xi)$ writes
\begin{equation}
\label{eq:structure1}
\forall y \in Q_N, \ \forall \xi \in \RR^d, \quad
W(y,\omega,\xi) = 
{\cal A}\left(y,X_1(\omega),\ldots,X_n(\omega),\xi\right) \quad \text{a.s.},
\end{equation}
where $\left\{ X_k(\omega) \right\}_{1 \leq k \leq n}$ are independent scalar
random variables, which are all distributed according to the uniform law
${\cal U}[0,1]$. In general, the function ${\cal A}$, as well as the
number $n$ of independent, identically distributed variables involved
in~\eqref{eq:structure1}, depend on $N$, the size of $Q_N$, although
this dependency is not made explicit in~\eqref{eq:structure1}. 

Second, we assume that the function ${\cal A}$ in~\eqref{eq:structure1} is
such that, for all $y \in Q_N$ and all $\xi \in \RR^d$, the map
\begin{equation}
\label{eq:structure2}
(x_1,\ldots,x_n) \in \RR^n \mapsto {\cal A}(y,x_1,\ldots,x_n, \xi)
\end{equation}
is non-decreasing with respect to each of its arguments. 

\medskip

\begin{proposition}
\label{prop:strong}
We assume~\eqref{eq:structure1}--\eqref{eq:structure2}. Let
$W^\star_N(\omega,\xi)$ be the approximated homogenized energy density
field defined by~\eqref{eq:correc-random_N}.
We define on $Q_N$ the field
$$
W^{\rm ant}(y, \omega, \xi) := 
{\cal A}(y,1-X_1(\omega),\ldots,1-X_n(\omega),\xi),
$$
antithetic to $W$ defined by~\eqref{eq:structure1}.
We associate to this field the approximate homogenized energy density
field $W^{{\rm ant},\star}_N(\omega,\xi)$, defined
by~\eqref{eq:correc-random_N} (replacing $W$ by $W^{\rm ant}$). Set
\begin{equation}
\label{eq:def_a_tilde}
\widetilde{W}^\star_N(\omega,\xi) := \frac12 
\left( W^\star_N(\omega,\xi) +  W^{{\rm ant},\star}_N(\omega,\xi) \right).
\end{equation}
Then, for any $\xi \in \RR^d$,
\begin{equation}
\label{eq:resu0_1}
\EE\left[\widetilde{W}^\star_N(\cdot,\xi) \right] 
= 
\EE\left[W^\star_N(\cdot,\xi)\right]
\ \ \text{and} \ \
\Var \left[\widetilde{W}^\star_N(\cdot,\xi) \right] 
\leq 
\frac12 \Var \left[W^\star_N(\cdot, \xi) \right].
\end{equation}
Otherwise stated, $\widetilde{W}^\star_N(\omega,\xi)$ is a random
variable which has the same expectation as $W^\star_N(\omega,\xi)$, and
its variance is smaller than half of that of $W^\star_N(\omega,\xi)$. 
\end{proposition}

As mentioned above, this result generalizes~\cite[Proposition 2.1]{mprf}
to the nonlinear convex variational setting considered here.

\medskip

Before proceeding, we briefly explain the usefulness of the above
result for variance reduction techniques. Assume we want to compute the
expectation of $W^\star_N(\omega,\xi)$, for some fixed vector $\xi \in
\RR^d$. Following the classical Monte-Carlo method recalled in
Section~\ref{sec:questions}, we estimate $\EE\left[ W^\star_N(\cdot,\xi)
\right]$ by its empirical mean. To this end, we consider $2M$
independent, identically distributed copies
$\left\{ W_m(y,\omega,\xi) \right\}_{1 \leq m \leq 2M}$ of the random field 
$W(y,\omega,\xi)$ on $Q_N$. To each copy $W_m$,
we associate an approximate homogenized energy density $W^{\star,m}_N(\omega,\xi)$,
defined by~\eqref{eq:correc-random_N}. We
next introduce the empirical mean 
\begin{equation}
\label{eq:estim1}
I_{2M} = \frac{1}{2M} 
\sum_{m=1}^{2M} W^{\star,m}_N(\omega,\xi),
\end{equation}
and consider that, in practice, the mean $\EE\left[ W^\star_N(\cdot,\xi)
\right]$ is equal to the estimator $I_{2M}$ within an approximate
margin of error  
$\dps 1.96 \frac{\sqrt{\Var \left[ W^\star_N(\cdot,\xi) \right]}}{\sqrt{2M}}$. 

\medskip

Alternate to considering~\eqref{eq:estim1}, we may consider
\begin{equation}
\label{eq:estim2}
\widetilde{I}_{2M} = 
\frac{1}{M} 
\sum_{m=1}^{M} \widetilde{W}^{\star,m}_N(\omega,\xi),
\end{equation}
where $\widetilde{W}^{\star,m}_N$ is defined by~\eqref{eq:def_a_tilde}.
Again, in practice, the mean
$\dps \EE\left[ W^\star_N(\cdot,\xi) \right] = 
\EE\left[ \widetilde{W}^\star_N(\cdot,\xi) \right]$ is equal to
$\widetilde{I}_{2M}$ within an
approximate margin of error $\dps 1.96 
\frac{\sqrt{\Var \left[ \widetilde{W}^\star_N(\cdot,\xi) \right]}}{\sqrt{M}}$. 
Observe now that both estimators~\eqref{eq:estim1} and~\eqref{eq:estim2}
are of equal cost, since they require the same number $2M$ of 
corrector problems to be solved. The
accuracy of the latter is better if and only if  
$\dps \Var\left[ \widetilde{W}^\star_N(\cdot,\xi) \right] \leq \frac12 
\Var\left[ W^\star_N(\cdot, \xi) \right]$, which is exactly the
bound~\eqref{eq:resu0_1} of Proposition~\ref{prop:strong}. 

\subsubsection{Variance reduction on the first derivative of the homogenized
  energy density}
\label{sec:1}

Restricting ourselves to the one-dimensional setting, we now state a
variance reduction result for the estimation of 
$\EE \left[ \xi \partial_\xi W^\star_N(\cdot,\xi) \right]$.
Note that, to distinguish derivatives with respect to $y$ from
derivatives with respect to $\xi$, we keep the notation $\partial_\xi
W$, even though we are in the one-dimensional situation.

We again make the structure assumption~\eqref{eq:structure1}, and
observe that it implies that
$$
\forall y \in (-N,N), \ \forall \xi \in \RR, \quad
\xi \partial_\xi W(y,\omega,\xi) 
= 
{\cal A}_1\left(y,X_1(\omega),\ldots,X_n(\omega),\xi\right) \quad \text{a.s.},
$$
where $\left\{ X_k(\omega) \right\}_{1 \leq k \leq n}$ are scalar
i.i.d. random variables, which are all distributed according to the uniform law
${\cal U}[0,1]$, and 
where the function ${\cal A}_1$, defined on $(-N,N) \times \RR^n \times
\RR$, is given by 
\begin{equation}
\label{eq:def_A1}
{\cal A}_1(y,x,\xi) = 
\xi \partial_\xi{\cal A}(y,x,\xi).
\end{equation} 
In addition, we assume that, for all $y \in (-N,N)$ and all $\xi \in
\RR$, the map 
\begin{equation}
\label{eq:structure-der2}
(x_1,\ldots,x_n) \in \RR^n \mapsto {\cal A}_1(y,x_1,\ldots,x_n, \xi)
\end{equation}
is non-decreasing with respect to each of its arguments. 

We recall that the function $\xi \mapsto W(y,\omega,\xi)$ is strictly
convex (see assumption~\eqref{eq:hyp_strict_convex}) and
satisfies~\eqref{eq:Caratheodory}. It
therefore has a unique minimizer $\xi_0(y,\omega)$. 
In the sequel, we consider energy densities such that this minimizer 
is independent of $y$ and $\omega$. Without loss of generality, we can assume
that $\xi_0 = 0$. We thus consider energy densities $W$ such that
\begin{equation}
\label{hyp:0-min-sto}
\text{$\xi \mapsto W(y,\omega,\xi)$ attains its minimum at $\xi = 0$,
  a.e. and a.s.} 
\end{equation}

\begin{proposition}
\label{prop:var-red-der}
Let $d=1$, and
assume~\eqref{eq:structure1},~\eqref{eq:def_A1},~\eqref{eq:structure-der2}
and~\eqref{hyp:0-min-sto}. 
We introduce
\begin{equation}
\label{eq:def-tilde-der1}
\widetilde{\xi  \partial_\xi W^\star_N} (\omega,\xi) 
:= 
\frac12 \left( 
\xi \partial_\xi W^{{\rm ant},\star}_N \left(\omega,\xi \right) 
+ 
\xi \partial_\xi W^\star_N \left( \omega, \xi \right) \right),
\end{equation}
where $W^{{\rm ant},\star}_N \left(\omega,\xi \right)$ and  
$W^\star_N \left( \omega, \xi \right)$ are defined as in
Proposition~\ref{prop:strong}. Then, for any $\xi \in \RR$,
\begin{equation}
\label{eq:resu_der}
\EE \left[ \widetilde{\xi \partial_\xi W^\star_N}(\cdot,\xi) \right] 
=  
\EE \left[ \xi \partial_\xi W^\star_N(\cdot,\xi) \right]
\quad \text{and} \quad
\Var \left[ \widetilde{\xi \partial_\xi W^\star_N}(\cdot,\xi) \right] 
\leq 
\frac12 \Var \left[ \xi \partial_\xi W^\star_N(\cdot,\xi) \right].
\end{equation}
\end{proposition}

\subsubsection{Variance reduction on the second derivative of the
  homogenized energy density}
\label{sec:2}

Considering again the one-dimensional setting as in
Section~\ref{sec:1}, we eventually state a
variance reduction result for the estimation of 
$\EE \left[ \partial^2_\xi W^\star_N(\cdot,\xi) \right]$.

Recall that, for any $y$ and $\omega$, the map $\xi \mapsto \partial_\xi
W(y,\omega,\xi)$ is increasing. We can therefore introduce its reciprocal
function $\zeta \mapsto \psi(y,\omega,\zeta)$, which is also
increasing. 

We again make the structure assumption~\eqref{eq:structure1}, and
observe that it implies that, for any $y \in (-N,N)$ and any $\zeta \in
\RR$, 
$$
\partial^2_\xi W\left(y, \omega, \psi(y, \omega, \zeta) \right) 
= 
{\cal A}_2(y,X_1(\omega),\ldots,X_n(\omega),\zeta) \quad \text{a.s.},
$$
where $\left\{ X_k(\omega) \right\}_{1 \leq k \leq n}$ are scalar
i.i.d. random variables, which are all distributed according to the uniform law
${\cal U}[0,1]$, and where the function ${\cal A}_2$, defined on $(-N,N)
\times \RR^n \times \RR$, is given by
\begin{equation}
\label{eq:def_A2}
{\cal A}_2(y,x,\zeta) = 
\partial^2_\xi{\cal A} \left(
y,x,\left[\partial_\xi{\cal A} (y,x,\cdot)\right]^{-1} (\zeta) \right),
\end{equation}
where $\zeta \mapsto \left[\partial_\xi{\cal A} (y,x,\cdot)\right]^{-1} (\zeta)$
is the reciprocal function of 
$\xi \mapsto \partial_\xi{\cal A} (y,x,\xi)$.

In addition, we assume that, for all $y \in (-N,N)$ and all $\zeta \in
\RR$, the map 
\begin{equation}
\label{eq:structure-der22}
(x_1,\ldots,x_n) \in \RR^n \mapsto {\cal A}_2(y,x_1,\ldots,x_n, \zeta)
\end{equation}
is non-decreasing with respect to each of its arguments. 

\begin{proposition}
\label{prop:var-red-der2}
Let $d=1$, and
assume~\eqref{eq:structure1},~\eqref{eq:def_A1},~\eqref{eq:structure-der2},~\eqref{eq:def_A2}
and~\eqref{eq:structure-der22}. We also assume
that~\eqref{hyp:0-min-sto} holds, and that
\begin{equation}
\label{eq:struct-monotony-sec-der-sto}
\begin{array}{c}
\xi \mapsto \partial^2_\xi W \left(y,\omega,\xi \right) \text{is non
  decreasing for $\xi \geq 0$}
\\
\text{and non increasing for $\xi \leq 0$, a.e. and a.s.}
\end{array}
\end{equation}
We introduce
$$
\widetilde{\partial^2_\xi W^\star_N} (\omega,\xi) 
:= 
\frac12 \left( 
\partial^2_\xi W^{{\rm ant},\star}_N \left(\omega,\xi \right) 
+ 
\partial^2_\xi W^\star_N \left( \omega, \xi \right) \right),
$$
where $W^{{\rm ant},\star}_N \left(\omega,\xi \right)$ and  
$W^\star_N \left( \omega, \xi \right)$ are defined as in
Proposition~\ref{prop:strong}. Then, for any $\xi \in \RR$,
\begin{equation}
\label{eq:resu_der2}
\EE \left[ \widetilde{\partial^2_\xi W^\star_N}(\cdot,\xi) \right] 
=  
\EE \left[ \partial^2_\xi W^\star_N(\cdot,\xi) \right]
\quad \text{and} \quad
\Var \left[ \widetilde{\partial^2_\xi W^\star_N}(\cdot,\xi) \right] 
\leq 
\frac12 \Var \left[ \partial^2_\xi W^\star_N(\cdot,\xi) \right].
\end{equation}
\end{proposition}
The density $W(y,\omega,\xi) = a(y,\omega) |\xi|^p$, where $a$ is positive
and bounded away from zero and $p \geq 2$, typically satisfies the
assumption~\eqref{eq:struct-monotony-sec-der-sto}. 

\subsection{Classical results on antithetic variables}
\label{sec:antithetic}

We first recall the following lemma, and provide its proof for
consistency. This result is crucial for our proof of
variance reduction using the technique of antithetic variables,
performed in Section~\ref{sec:proof}. 

\begin{lemma}[\cite{liu}, page 27]
\label{lem:lapeyre}
Let $f$ and $g$ be two real-valued functions defined on~$\RR^n$, which
are non-decreasing with respect to each of their arguments. Consider 
$X = (X_1,\ldots,X_n)$
a vector of random variables, which are all independent from one another. Then
\begin{equation}
\label{eq:cov_neg}
\Cov(f(X),g(X)) \geq 0.
\end{equation}
\end{lemma}

\begin{proof}
This lemma is proved by induction. We treat the one-dimensional case
($n=1$) below, and we refer to~\cite[Proof of Lemma 2.1]{mprf} for the
induction. 
Consider $X$ and $Y$ two independent scalar random variables, identically
distributed. Both functions $f$ and $g$ are
non-decreasing, so
$$
(f(X) - f(Y)) \ (g(X) - g(Y)) \geq 0.
$$
We now take the expectation of the above inequality:
$$
\EE (f(X) \ g(X) ) + \EE (f(Y) \ g(Y) ) \geq 
\EE (f(Y) \ g(X)) + \EE (f(X) \ g(Y)).
$$
As $X$ and $Y$ share the same law, and are independent, this yields
$$
\EE (f(X) \ g(X) ) \geq \EE (f(X)) \ \EE(g(X)),
$$
and~\eqref{eq:cov_neg} follows for $n=1$. 
\end{proof}

The following result is a simple consequence of the above lemma (see
e.g.~\cite{mprf} for a proof).

\begin{corollary}[\cite{liu}]
\label{cor:antithetic}
Let $f$ be a function defined on $\RR^n$, which is non-decreasing
with respect to each of its arguments. Consider $X = (X_1,\ldots,X_n)$
a vector of random variables, which are all independent from one
another, and distributed according to the uniform law \ 
${\cal U}[0,1]$. Then
$$
\Var \left( \frac12 \left( f(X) + f(1-X) \right) \right) \leq 
\frac12 \Var \left( f(X) \right),
$$
where we denote $1-X = (1-X_1,\ldots,1-X_n) \in \RR^n$. 
\end{corollary}

\begin{proof}
Choosing $g(x_1,\ldots,x_n) = -f(1-x_1,\ldots,1-x_n)$ in 
Lemma~\ref{lem:lapeyre}, we obtain that
$$
\Cov(f(X),f(1-X)) = 
\Cov(f(X_1,\ldots,X_n),f(1-X_1,\ldots,1-X_n)) \leq 0.
$$
We next observe that  
\begin{eqnarray*}
\Var \left( \frac12 \left( f(X) + f(1-X) \right) \right)
&=&
\frac12 \Var(f(X)) + 
\frac12 \Cov \left( f(X),f(1-X) \right)
\\
&\leq&
\frac12 \Var(f(X)),
\end{eqnarray*}
where we have used that $\Var(f(X)) = \Var(f(1-X))$.
\end{proof}


\subsection{Derivatives of the corrector and of the homogenized energy
  density}
\label{sec:expression}

We now introduce the correctors as the solutions
to~\eqref{eq:correc-random_N}: 
$$
w^N(\cdot, \omega, \xi) := \text{arginf} \left\{ \int_{Q_N} W(\cdot,
  \omega, \xi + \nabla v), \quad v \in W^{1,p}_\# (Q_N), 
\quad \int_{Q_N} v = 0 \right\}.
$$
In this section, we derive some useful expressions for the derivatives
with respect to $\xi$ of $w^N$ and of $W^\star_N$.

The first order optimality condition in~\eqref{eq:correc-random_N} reads
\begin{equation}
\label{eq:optimality}
\forall h \in W^{1,p}_\# (Q_N), \quad 
\int_{Q_N} \left( \nabla h \right)^T \partial_\xi W \left(\cdot, \omega,
  \xi + \nabla w^N \right) = 0. 
\end{equation}
We deduce from that condition that 
\begin{equation}
\label{eq:Wstar-der-1}
\partial_\xi W^\star_N(\omega, \xi) 
= \frac{1}{|Q_N|} 
\int_{Q_N} \partial_\xi W\left(\cdot, \omega, \xi + \nabla w^N \right),
\end{equation}
and we note that we do not need to know $\partial_\xi w^N$ to compute
$\partial_\xi W^\star_N$. Computing the derivative of this equality with
respect to $\xi$, we obtain that
\begin{equation}
\label{eq:Wstar-der-2}
\partial^2_\xi W^\star_N(\omega, \xi) 
= \frac{1}{|Q_N|} 
\int_{Q_N} \left( \Id + \partial_\xi \nabla w^N \right) 
\partial^2_\xi W\left(\cdot, \omega, \xi + \nabla w^N \right)
\end{equation}
with the convention that 
$\dis \left[ \partial_\xi \nabla w^N \right]_{jk} = 
\frac{\partial^2 w^N}{\partial \xi_j \partial y_k}$
for $1 \leq j,k \leq d$. We can actually
obtain a somewhat more symmetric expression. Computing the derivative
of~\eqref{eq:optimality} with respect to $\xi$, we indeed see that
\begin{equation}
\label{eq:optimality_2}
\forall h \in W^{1,p}_\# (Q_N), \quad 
\int_{Q_N} \left( \Id + \partial_\xi \nabla w^N \right) 
\partial^2_\xi W\left(\cdot, \omega, \xi + \nabla w^N \right) \nabla h
= 0.
\end{equation}
We then infer from~\eqref{eq:Wstar-der-2} and~\eqref{eq:optimality_2}
that
\begin{equation}
\label{eq:second_derivative}
\partial_\xi^2 W^\star_N (\omega, \xi) 
= 
\frac{1}{|Q_N|} \int_{Q_N} \left( \Id + \partial_\xi \nabla w^N \right) 
\partial_\xi^2 W\left(\cdot, \omega, \xi+\nabla w^N\right) 
\left( \Id + \partial_\xi \nabla w^N \right)^T.
\end{equation}

\begin{remark}
Using the same kind of arguments, we see that the function 
$\dis g_j = \frac{\partial w^N}{\partial \xi_j} \in W^{1,p}_\# (Q_N)$ is
solution to the variational formulation
\begin{multline}
\label{eq:def_g_j}
\forall h \in W^{1,p}_\# (Q_N), \quad 
\int_{Q_N} \left( \nabla h \right)^T 
\partial^2_\xi W\left(\cdot, \omega, \xi + \nabla w^N \right)
\nabla g_j 
\\ =
- \sum_{i=1}^d \int_{Q_N} \frac{\partial h}{\partial y_i}
\frac{\partial^2 W}{\partial \xi_j \partial \xi_i}
\left(\cdot, \omega, \xi + \nabla w^N \right).
\end{multline}
Suppose that $W$ is $\alpha$-convex
(i.e. satisfies~\eqref{eq:hyp_alpha_convex}). Then
problem~\eqref{eq:def_g_j} is well-posed and allows to uniquely
determine (up to an additive constant) $g_j$, by solving 
a linear elliptic partial differential equation. 

Combined
with~\eqref{eq:second_derivative}, this remark provides a practical way
to compute $\partial_\xi^2 W^\star_N (\omega, \xi)$ without using any
finite difference approximation in $\xi$. 
\end{remark}
We finally note that, in view of~\eqref{eq:Wstar-der-1}, we have
\begin{equation}
\label{eq:first_axial}
\xi \cdot \partial_\xi W^\star_N(\omega, \xi) 
= \frac{1}{|Q_N|} 
\int_{Q_N} 
\xi \cdot \partial_\xi W\left(\cdot, \omega, \xi + \nabla w^N \right).
\end{equation}
Likewise, in view of~\eqref{eq:second_derivative}, we see that
\begin{multline}
\label{eq:second_axial}
\xi^T \partial_\xi^2 W^\star_N (\omega, \xi) \xi
= 
\\
\frac{1}{|Q_N|} \int_{Q_N} 
\left[ \xi + \nabla \left( \xi \cdot \partial_\xi w^N \right) \right]^T
\partial_\xi^2 W\left(\cdot, \omega, \xi+\nabla w^N\right) 
\left[ \xi + \nabla \left( \xi \cdot \partial_\xi w^N \right) \right].
\end{multline}

\subsection{Monotonicity properties}
\label{sec:monoton}

Our goal in this section is to establish monotonicity properties for the
homogenization process. Such properties are indeed useful to apply
Corollary~\ref{cor:antithetic} and therefore prove variance reduction.

To simplify the notation, we assume in this section that we are in a {\em periodic}
setting. For any $\xi \in \RR^d$, the function $y \mapsto W(y,\xi)$ is
supposed to be $Q$-periodic (with $Q=(0,1)^d$), to satisfy the growth
condition~\eqref{eq:Caratheodory} and to be strictly convex with respect to
$\xi$.
The associated homogenized energy density is then given by
\begin{equation}
\label{eq:champ-homog_per}
W^\star(\xi) = \inf \left\{ 
\int_Q W\left(y,\xi + \nabla w(y) \right) \, dy,
\quad w \in W^{1,p}_\#(Q), \quad \int_Q w = 0 \right\}.
\end{equation}

We first show a monotonicity property on the homogenized energy density
in Section~\ref{sec:monotony}. Next, restricting ourselves to the
one-dimensional setting, we show monotonicity properties for the first
and the second derivative of the homogenized energy density (see
respectively Sections~\ref{sec:monotony_1} and~\ref{sec:monotony_2}).

\subsubsection{On the homogenized energy density}
\label{sec:monotony}

The following result is an extension to the nonlinear setting of a
well-known result in the linear setting (see~\cite[page
12]{composites}). 

\begin{lemma}
\label{lem:monotony}
Suppose that the fields $W_1$ and $W_2$ satisfy
\begin{equation}
\label{eq:W_croiss}
\forall \xi \in \RR^d, \quad
W_2(y, \xi) \geq W_1(y, \xi) \text{ a.e. on $Q$}. 
\end{equation}
We denote $W_1^\star$ and $W_2^\star$ the corresponding homogenized
energy densities, defined by~\eqref{eq:champ-homog_per}. We then have
\begin{equation}
\label{eq:Wstar_croiss}
\forall \xi \in \RR^d, \quad W^\star_2(\xi) \geq W^\star_1(\xi).
\end{equation}
\end{lemma}

\begin{proof}
Fix $\xi \in \RR^d$. For any $v \in W^{1,p}_\#(Q)$ with $\dis \int_Q v =
0$, we have that
$$
W^\star_1(\xi)
\leq
\int_Q W_1 \left(y,\xi + \nabla v(y) \right) \, dy
\leq 
\int_Q W_2 \left(y,\xi + \nabla v(y) \right) \, dy.
$$
Taking the infimum over $v$, we obtain the claimed result. 
\end{proof}

\begin{remark}
\label{rem:mon-hom}

Consider the case of an energy density that is positively homogeneous of
degree $p$ with respect to its variable $\xi$, that is such that
$W(y,\lambda \xi) = |\lambda|^p \, W(y,\xi)$ for any $y \in \RR^d$, $\xi
\in \RR^d$ and $\lambda \in \RR$. A typical example is $\dis W(y,\xi) =
\frac{1}{p} a(y) |\xi|^p$. We then have, for any $y$ and $\xi$, that
\begin{equation}
\label{eq:utile}
\xi \cdot \partial_\xi W(y,\xi) = p W(y,\xi)
\quad \text{and} \quad 
\xi^T \partial^2_\xi W(y,\xi) \xi = p (p-1)W(y,\xi).
\end{equation} 
Using successively~\eqref{eq:first_axial},~\eqref{eq:optimality}
and~\eqref{eq:utile}, we obtain that
\begin{eqnarray}
\xi \cdot \partial_\xi W^\star(\xi) 
&=&  
\int_Q 
\xi \cdot \partial_\xi W\left(\cdot, \xi + \nabla w \right)
\nonumber
\\
&=&
\int_Q 
(\xi + \nabla w) \cdot \partial_\xi W\left(\cdot, \xi + \nabla w \right)
\nonumber
\\
&=&
p \int_Q W\left(\cdot, \xi + \nabla w \right)
\nonumber
\\
&=&
p W^\star(\xi),
\label{eq:Wstar-der-xi-1}
\end{eqnarray}
where $w$ is the corrector, solution to~\eqref{eq:champ-homog_per}. 

We next observe that, for any $\lambda \in \RR$, we have 
$w(\cdot,\lambda \xi) = \lambda w(\cdot,\xi)$. Thus, for any $y$, the
map $\xi \mapsto w(y,\xi)$ is homogeneous of degree one, and therefore
$\xi \cdot \partial_\xi w = w$. We thus infer from~\eqref{eq:second_axial},
using~\eqref{eq:utile}, that 
\begin{eqnarray}
\xi^T \partial_\xi^2 W^\star (\xi) \xi 
&=& 
\int_Q [\xi + \nabla w]^T \partial_\xi^2 W(\cdot, \xi+\nabla w) 
[\xi + \nabla w]
\nonumber
\\
&=&
p(p-1) \int_Q W(\cdot, \xi+\nabla w) 
\nonumber
\\
&=& p(p-1) W^\star(\xi).
\label{eq:sec-tr-simplified}
\end{eqnarray}
Consider now two fields $W_1$ and $W_2$ that are positively homogeneous of
degree $p$ with respect to the variable $\xi$ and
satisfy~\eqref{eq:W_croiss}. Then we deduce
from~\eqref{eq:Wstar_croiss},~\eqref{eq:Wstar-der-xi-1}
and~\eqref{eq:sec-tr-simplified} that, for all $\xi \in \RR^d$, 
$$
\xi \cdot \partial_\xi W_2^\star(\xi) 
\geq
\xi \cdot \partial_\xi W_1^\star(\xi) 
\quad \text{and} \quad
\xi^T \partial_\xi^2 W_2^\star (\xi) \xi
\geq
\xi^T \partial_\xi^2 W_1^\star (\xi) \xi.
$$
\end{remark}

\subsubsection{On the first derivative of the homogenized energy density}
\label{sec:monotony_1}

We now establish a monotonicity result on the derivative of
$W^\star(\xi)$, in the one-dimensional setting. 

As in Section~\ref{sec:1} (see~\eqref{hyp:0-min-sto}), we consider
energy densities $W$ such that
\begin{equation}
\label{hyp:0-min}
\text{$\xi \mapsto W(y,\xi)$ attains its minimum at $\xi = 0$ for almost
  all $y \in Q$.} 
\end{equation}

\begin{lemma}
\label{lem:mon-der-one-d}
Let $d=1$, and consider two energy densities $W_1$ and $W_2$
satisfying~\eqref{hyp:0-min}, and such that
\begin{equation}
\label{eq:W1_croiss}
\forall \xi \in \RR, \quad
\xi \partial_\xi W_2(y,\xi) \geq \xi \partial_\xi W_1(y,\xi)
\text{ a.e. on $(0,1)$}. 
\end{equation}
We denote $W_1^\star$ and $W_2^\star$ the corresponding homogenized
energy densities, defined by~\eqref{eq:champ-homog_per}. We then have
\begin{equation}
\label{eq:Wstar1_croiss}
\forall \xi \in \RR, \quad 
\xi \partial_\xi W_2^\star(\xi) \geq \xi \partial_\xi W_1^\star(\xi).
\end{equation}
\end{lemma}

\begin{proof}
We first claim that 
\begin{equation}
\label{eq:claim1}
\text{
$\partial_\xi W^\star(\xi)$ has the same sign as $\xi$.}
\end{equation}
To prove this, we note that the corrector equation reads
(see~\eqref{eq:optimality}) 
$$
\frac{d}{dy} \left[ \partial_\xi 
W \left(y, \xi + \frac{dw}{dy}(y,\xi) \right) \right] =
0 \quad \text{on $(0,1)$}, \quad w(\cdot,\xi) \text{ is 1-periodic}.
$$
We therefore see that 
$\dis \partial_\xi W \left(y, \xi + \frac{dw}{dy}(y,\xi)\right)$ is
independent of $y$, and using~\eqref{eq:Wstar-der-1}, we obtain that
$$
\partial_\xi W \left(y, \xi + \frac{dw}{dy}(y,\xi)\right) = 
\partial_\xi W^\star(\xi)
\quad \text{on $(0,1)$}.
$$
Let $\xi \mapsto \psi(y,\xi)$ be the reciprocal function of 
$\xi \mapsto \partial_\xi W(y,\xi)$,
which exists and is increasing thanks to the
strict convexity of $\xi \mapsto W(y, \xi)$.
We deduce from the above equation,
after integration over $(0,1)$, that
\begin{equation}
\label{eq:one-d-K}
\xi = \int_0^1 \psi\left(y, \partial_\xi W^\star(\xi) \right) dy.
\end{equation}
We are now in position to prove~\eqref{eq:claim1}.
Indeed, we first note that~\eqref{hyp:0-min}, that reads
$\partial_\xi W(y,\xi=0) = 0$, implies that $\psi(y,0) = 0$. If
$\partial_\xi W^\star(\xi) \geq 0$, then $\psi(y,\partial_\xi
W^\star(\xi) \geq \psi(y,0) = 0$, hence, integrating over $(0,1)$ and
using~\eqref{eq:one-d-K}, we obtain $\xi \geq 0$. Likewise, 
$\partial_\xi W^\star(\xi) \leq 0$ implies that $\xi \leq 0$. The
claim~\eqref{eq:claim1} is proved. 

\medskip

To proceed, we see that the assumption~\eqref{eq:W1_croiss} equivalently
reads, using the reciprocal functions, 
\begin{equation}
\label{eq:hyp_equiv}
\forall \zeta \in \RR, \quad
\zeta \psi_2(y,\zeta) \leq \zeta \psi_1(y,\zeta)
\quad \text{a.e. on $(0,1)$}.
\end{equation}
We now prove~\eqref{eq:Wstar1_croiss} by contradiction. Assume that 
$\xi \partial_\xi W_2^\star(\xi) < \xi \partial_\xi W_1^\star(\xi)$ for
some $\xi \in \RR$. Without loss of generality, we can assume that $\xi
> 0$, and therefore $\partial_\xi W_2^\star(\xi) < \partial_\xi
W_1^\star(\xi)$. Using~\eqref{eq:claim1}, we additionally have 
$0 < \partial_\xi W_2^\star(\xi)$.
Using that $\zeta \mapsto \psi_2(y,\zeta)$ is
increasing and~\eqref{eq:hyp_equiv} with $\zeta = \partial_\xi
W_1^\star(\xi) > 0$, we have
$$
\psi_2\left( y, \partial_\xi W_2^\star(\xi) \right) 
< 
\psi_2\left( y, \partial_\xi W_1^\star(\xi) \right) 
\leq 
\psi_1\left( y, \partial_\xi W_1^\star(\xi) \right).
$$
Integrating over $(0,1)$ and using~\eqref{eq:one-d-K} yields
$$
\xi 
= 
\int_0^1 \psi_2 \left( y, \partial_\xi W_2^\star(\xi) \right) dy 
< 
\int_0^1 \psi_1 \left( y, \partial_\xi W_1^\star(\xi) \right) dy 
= 
\xi,
$$
and we reach a contradiction. This concludes the proof.
\end{proof}

\subsubsection{On the second derivative of the homogenized energy density}
\label{sec:monotony_2}

We next turn to monotonicity properties of the second derivative of the
homogenized energy density. As in Section~\ref{sec:monotony_1}, we
consider energy densities satisfying~\eqref{hyp:0-min}. We additionally
request that, almost everywhere in $(0,1)$,
\begin{equation}
\label{eq:struct-monotony-sec-der}
\begin{array}{c}
\xi \mapsto \partial^2_\xi W \left( y,\xi \right) \text{is non
  decreasing for $\xi \geq 0$}
\\
\text{and non increasing for $\xi \leq 0$.}
\end{array}
\end{equation}

\begin{lemma}
\label{lem:mon-der2-one-d}
Let $d=1$, and consider two energy densities $W_1$ and $W_2$
satisfying~\eqref{hyp:0-min},~\eqref{eq:W1_croiss},~\eqref{eq:struct-monotony-sec-der}
and such that 
\begin{equation}
\label{eq:W2_croiss}
\forall \zeta \in \RR, \quad
\partial^2_\xi W_2 \left( y, \psi_2(y,\zeta) \right) 
\geq  
\partial^2_\xi W_1 \left( y, \psi_1(y,\zeta) \right)
\text{ a.e. on $(0,1)$}. 
\end{equation}
We denote $W_1^\star$ and $W_2^\star$ the corresponding homogenized
energy densities, defined by~\eqref{eq:champ-homog_per}. We then have
\begin{equation}
\label{eq:Wstar2_croiss}
\forall \xi \in \RR, \quad 
\partial^2_\xi W_2^\star(\xi) \geq \partial^2_\xi W_1^\star(\xi).
\end{equation}
\end{lemma}
We recall that $\zeta \mapsto \psi(y,\zeta)$ is the reciprocal function
of $\xi \mapsto \partial_\xi W(y,\xi)$.

\begin{proof}
We first compute the derivative of~\eqref{eq:one-d-K} and obtain
\begin{equation}
\label{eq:one-d-K'}
\frac{1}{\partial^2_\xi W^\star(\xi)} 
= 
\int_0^1 \frac{dy}{\partial^2_\xi 
W \left[y, \psi\left(y,\partial_\xi W^\star(\xi) \right) \right]}.
\end{equation}
It is sufficient to prove~\eqref{eq:Wstar2_croiss} for $\xi >
0$. Using~\eqref{eq:Wstar1_croiss} and the fact that $\psi_1$ and
$\partial^2_\xi W_1$ are non-decreasing with respect to their second
argument, we have
$$
\partial^2_\xi W_1 \left( y, \psi_1\left(y, \partial_\xi W_1^\star(\xi)
  \right) \right) 
\leq 
\partial^2_\xi W_1 \left( y, \psi_1\left(y, \partial_\xi W_2^\star(\xi)
  \right) \right).
$$
Using~\eqref{eq:W2_croiss} for $\zeta = \partial_\xi W_2^\star(\xi)$, we
deduce that
$$
\partial^2_\xi W_1 \left( y, \psi_1\left(y, \partial_\xi W_1^\star(\xi)
  \right) \right) 
\leq
\partial^2_\xi W_2 \left( y, \psi_2\left(y, \partial_\xi W_2^\star(\xi)
  \right) \right).
$$
In view of~\eqref{eq:one-d-K'}, this inequality readily
implies~\eqref{eq:Wstar2_croiss} for $\xi > 0$. This concludes the
proof. 
\end{proof}

\subsection{Proof of
  Propositions~\ref{prop:strong},~\ref{prop:var-red-der}
  and~\ref{prop:var-red-der2}}
\label{sec:proof}

Now that we have collected all the necessary ingredients, we are in
position to prove our main results. 

\subsubsection{Variance reduction on the homogenized energy density}

\begin{proof}[Proof of Proposition~\ref{prop:strong}]
As $1-X_k(\omega)$ and $X_k(\omega)$ share the same law, so do the
fields $W$ and $W^{\text{ant}}$ on $Q_N$. Hence, the homogenized fields
$W^\star_N(\omega,\xi)$ and $W^{{\rm ant},\star}_N(\omega, \xi)$ share
the same law, and we obtain the first assertion of~\eqref{eq:resu0_1}. 

We now choose a vector $\xi \in \RR^d$, and denote by $\mathcal{P}^\xi_N$ 
the operator that associates to a given $Q_N$-periodic energy density
the homogenized energy density evaluted at $\xi$. We see
from~\eqref{eq:correc-random_N} that $W^\star_N(\omega,\xi)$ is the
effective energy density (evaluated at $\xi$) obtained by periodic
homogenization of $W_{| y \in Q_N}$:
\begin{equation}
\label{eq:def_PNxi}
W^\star_N(\omega,\xi) = 
\mathcal{P}^\xi_N \left[ W(\cdot,\omega,\cdot)_{| y \in Q_N} \right]
\quad \text{a.s.}
\end{equation}
Using the function ${\cal A}$ of~\eqref{eq:structure1}, we
introduce the map
\begin{eqnarray*}
f : \RR^n & \to & \RR
\\
x & \mapsto & 
\mathcal{P}^\xi_N \left[ {\cal A}(\cdot, x,\cdot) \right],
\end{eqnarray*}
see that $f(X(\omega)) = W^\star_N(\omega,\xi)$ and that, using the
definition~\eqref{eq:def_a_tilde} of
$\widetilde{W}^\star_N(\omega,\xi)$, we have
\begin{equation}
\label{eq:titi2}
\frac12 \left( f(X(\omega)) + f(1-X(\omega)) \right) 
= 
\frac12 \left( W^\star_N(\omega,\xi) + W^{{\rm ant},\star}_N(\omega,\xi) \right)
=
\widetilde{W}^\star_N(\omega,\xi).
\end{equation}

We now infer from Assumption~\eqref{eq:structure2} that, for any
$y \in Q_N$ and any $\zeta \in \RR^d$, the function ${\cal A}(y, \cdot,
\zeta)$ is non-decreasing with respect to each of its arguments. In view
of Lemma~\ref{lem:monotony}, we obtain that $f$ is non-decreasing. 

\medskip

We are thus in position to use Corollary~\ref{cor:antithetic}, which
yields
$$
\Var \left( \frac12 \left( f(X) + f(1-X) \right) \right) \leq 
\frac12 \Var \left( f(X) \right).
$$
Using~\eqref{eq:titi2}, we obtain
$$
\Var \left( \widetilde{W}^\star_N(\cdot,\xi) \right)
\! = \!
\Var \left[ \frac12 \left( f(X) + f(1-X) \right) \right]
\! \leq \!
\frac12 \Var \left( f(X) \right)
\! = \! 
\frac12 \Var \left( W^\star_N(\cdot,\xi) \right),
$$
which concludes the proof of the second assertion of~\eqref{eq:resu0_1}
and of the Proposition~\ref{prop:strong}.
\end{proof}

\begin{remark}
\label{rem:mon-hom2}
Following Remark~\ref{rem:mon-hom}, consider a positively
homogeneous energy density $W$. We have shown there that $\xi \cdot
\partial_\xi W_N^\star(\omega,\xi)$ and $\xi^T \partial^2_\xi
W_N^\star(\omega,\xi) \xi$ are equal (up to a multiplicative constant) to
$W_N^\star(\omega,\xi)$. Thus, under
Assumptions~\eqref{eq:structure1}--\eqref{eq:structure2}, variance
reduction holds for these two outputs as well. 
\end{remark}

\subsubsection{Variance reduction on the first derivative of the homogenized
  energy density}

\begin{proof}[Proof of Proposition~\ref{prop:var-red-der}]
The proof follows the same lines as that of
Proposition~\ref{prop:strong}. 

As $1-X_k(\omega)$ and $X_k(\omega)$ share the same law, so do the
fields $W$ and $W^{\rm ant}$ on $Q_N$. Hence, the quantities
$\xi \partial_\xi W^\star_N(\omega, \xi)$ and $\xi \partial_\xi W^{{\rm
    ant},\star}_N(\omega, \xi)$ share the same law, which implies the
first assertion of~\eqref{eq:resu_der}.

To prove the second assertion, we again make use, as in the proof of
Proposition~\ref{prop:strong}, of the operator ${\mathcal P}^\xi_N$ 
that associates to a given $Q_N$-periodic energy density
the homogenized energy density evaluted at
$\xi$ (here, $Q_N = (-N,N)$). Expression~\eqref{eq:def_PNxi} holds. 
Choosing a vector $\xi \in \RR$, we introduce the function
\begin{eqnarray*}
f : \RR^n & \to & \RR
\\
x & \mapsto & 
\xi \partial_\xi \left[ {\mathcal P}^\xi_N \left({\cal A}(\cdot,
    x,\cdot)\right) \right],
\end{eqnarray*}
which obviously satisfies $f(X(\omega)) = 
\xi \partial_\xi W^\star_N(\omega,\xi)$.
Using the definition~\eqref{eq:def-tilde-der1} of
$\widetilde{\xi \partial_\xi W^\star_N}$, we have 
\begin{equation}
\label{eq:der1-proba}
\frac12 \left[ f(X(\omega)) + f(1-X(\omega)) \right] 
= 
\frac12 \left[ \xi  \partial_\xi W^\star_N(\omega,\xi) 
+ 
\xi \partial_\xi W^{{\rm ant},\star}_N(\omega,\xi) \right]
=
\widetilde{\xi  \partial_\xi W^*_N}(\omega,\xi).
\end{equation}

We now infer from Assumption~\eqref{eq:structure-der2} that, for any
$y \in (-N,N)$ and any $\zeta \in \RR$, the function ${\cal A}_1(y,
\cdot, \zeta)$ is non-decreasing with respect to each of its
arguments. In view of Lemma~\ref{lem:mon-der-one-d}, we thus obtain that
$f$ is non-decreasing. 

\medskip

Using Corollary~\ref{cor:antithetic}, we write that $\dis
\Var \left( \frac12 \left( f(X) + f(1-X) \right) \right) \leq 
\frac12 \Var \left( f(X) \right)$.
In view of~\eqref{eq:der1-proba}, we recast this inequality as
$$
\Var \left[ \widetilde{\xi \partial_\xi W^\star_N}(\cdot,\xi) \right]
\leq
\frac12 \Var \left( \xi \partial_\xi W^\star_N(\cdot,\xi) \right),
$$
and therefore obtain the second assertion of~\eqref{eq:resu_der}. This
concludes the proof of Proposition~\ref{prop:var-red-der}.
\end{proof}

\subsubsection{Variance reduction on the second derivative of the
  homogenized energy density}

\begin{proof}[Proof of Proposition~\ref{prop:var-red-der2}]

The proof follows the same lines as the proof of
Proposition~\ref{prop:var-red-der}. Using
Assumptions~\eqref{eq:structure-der2} and~\eqref{eq:structure-der22}, we
see that Assumptions~\eqref{eq:W1_croiss} and~\eqref{eq:W2_croiss} of
Lemma~\ref{lem:mon-der2-one-d} are satisfied. The monotonicity result of
Lemma~\ref{lem:mon-der2-one-d} next allows to use
Corollary~\ref{cor:antithetic}, which implies~\eqref{eq:resu_der2}.
\end{proof}

\subsection{Examples satisfying our structure assumptions}
\label{sec:examples}

Before proceeding to the numerical tests, we give here some specific
examples of fields $W$ that satisfy the above assumptions. We consider
the case
\begin{equation}
\label{eq:test-case}
W(y,\omega,\xi) = 
a(y, \omega) \frac{|\xi|^p}{p} + c(y, \omega) \frac{|\xi|^2}{2},
\quad p \geq 2,
\end{equation}
with $c(y,\omega) \geq 0$ and $a(y, \omega) \geq a_- > 0$ a.e. and a.s.,
and provide sufficient conditions on the scalar fields $a$ and $c$ for
the structure
assumptions~\eqref{eq:structure1},~\eqref{eq:structure2},~\eqref{eq:structure-der2}
and~\eqref{eq:structure-der22} to be satisfied.
Note that~\eqref{hyp:0-min-sto}
and~\eqref{eq:struct-monotony-sec-der-sto} are already fullfilled.

Consider two families 
$\left( a_k(\omega) \right)_{k \in \mathbb{Z}^d}$ and 
$\left( c_k(\omega)\right)_{k \in \mathbb{Z}^d}$ of
independent, identically distributed random variables, and assume that
\begin{equation}
\label{eq:random-form1}
a(y,\omega) = \sum_{k \in \mathbb{Z}^d} \mathbf{1}_{Q + k}(y)
a_k(\omega), 
\ \ 
c(y,\omega) = \sum_{k \in \mathbb{Z}^d} \mathbf{1}_{Q + k}(y) c_k(\omega),
\end{equation}
where $Q=(0,1)^d$ and $Q+k$ is the cube $Q$ translated by the 
vector $k \in \mathbb{Z}^d$. The scalar field $a(y,\omega)$ is therefore
constant in each cube $Q+k$ with i.i.d. values $a_k(\omega)$, and likewise
for $c(y,\omega)$. 

We assume that there exist $\alpha > 0$ and
$\beta < \infty$ such that, for all $k \in \ZZ^d$, 
$0 < \alpha \leq a_k(\omega) \leq \beta < +\infty$ and 
$0 \leq c_k(\omega) \leq \beta < +\infty$ almost surely. Consequently,~\eqref{eq:Caratheodory} holds. 

Introduce now the cumulative distribution functions $\dps P_a(x) = 
\nu_a(-\infty,x)$, where $\nu_a$ is the common probability measure of all the
$a_k$, and next the non-decreasing functions $f_a(x) = \inf\{z; P_a(x)
\geq z\}$. Then, for any random variable $X^a(\omega)$ uniformly
distributed in $[0,1]$, the random variable $f_a(X^a(\omega))$ is
distributed according to the measure $\nu_a$. As a consequence, we can
recast~\eqref{eq:random-form1} in the form 
$$
a(y,\omega) = \sum_{k \in \mathbb{Z}^d} 
\mathbf{1}_{Q + k}(y) f_a(X^a_k(\omega)), 
$$
where $\left( X^a_k(\omega)\right)_{k \in \mathbb{Z}^d}$ is a family of
independent random variables that are all uniformly distributed in
$[0,1]$, and $f_a$ is non-decreasing. We can proceed likewise for the
variables $c_k$. This yields an example
where~\eqref{eq:structure1},~\eqref{eq:structure2}
and~\eqref{eq:structure-der2} hold. In particular, the function ${\cal A}$
of~\eqref{eq:structure1} reads
$$
{\cal A}(y,x_a,x_c,\xi) = 
\frac{|\xi|^p}{p} \sum_{k \in I_N} \mathbf{1}_{Q + k}(y) f_a(x^a_k) 
+ 
\frac{|\xi|^2}{2}
\sum_{k \in I_N} \mathbf{1}_{Q + k}(y) f_c(x^c_k), 
$$
where $I_N = \left\{ k \in \ZZ^d \text{ s.t. } Q + k \subset Q_N
\right\}$ and $x_a = \left\{ x^a_k \right\}_{k \in I_N}$.
As shown in~\cite{mprf}, more
general fields $a(y, \omega)$ (where random variables may be correlated)
also fall into this framework.

\medskip

In what follows, we prove that, under assumptions~\eqref{eq:test-case}
and~\eqref{eq:random-form1}, and if $p \leq 3$, the structure
assumption~\eqref{eq:structure-der22} holds. Without loss of generality,
we may assume that $y \in (0,1)$, and write that
$$
\forall y \in (0,1), \quad
{\cal A}(x_a,x_c,\xi) = \bar a \frac{|\xi|^p}{p} + \bar c \frac{|\xi|^2}{2},
$$ 
with $\bar a = f_a(x^a_0)$ and $\bar c = f_c(x^c_0)$. By a slight abuse
of notation, we keep implicit the dependency with respect to $y$, work
with $\bar a$ and $\bar c$ rather than $x_a$ and $x_c$, and write
$$
{\cal A}(\bar a, \bar c,\xi) = \bar a \frac{|\xi|^p}{p} + \bar c
\frac{|\xi|^2}{2}.
$$ 
We compute 
$$
\partial_\xi {\cal A}(\bar a,\bar c,\xi) = \bar a |\xi|^{p-2} \xi + \bar c \xi
$$
and denote $\zeta \mapsto g(\bar a,\bar c,\zeta)$ the reciprocal to the
function $\xi \mapsto \partial_\xi {\cal A}(\bar a,\bar c,\xi)$:
$$
\zeta = \bar a \, \left| g(\bar a,\bar c,\zeta) \right|^{p-2} 
g(\bar a,\bar c,\zeta)
+ \bar c g(\bar a,\bar c, \zeta).
$$
The function ${\cal A}_2$ of~\eqref{eq:structure-der22} then reads
$$
{\cal A}_2(\bar a, \bar c,\zeta) = 
(p-1) \bar a \left| g(\bar a,\bar c,\zeta) \right|^{p-2} + \bar c.
$$
We are left with showing that ${\cal A}_2$ is non-decreasing with
respect to $\bar a$ and $\bar c$.

A first remark is that since $g(\bar a,\bar c,\zeta)$ has the same sign
as $\zeta$ (recall that $\bar a>0$ and $\bar c \geq 0$), we may as well
restrict ourselves to $\zeta > 0$ and $g(\bar a, \bar c,\zeta) > 0$. We
hence have
\begin{eqnarray}
\nonumber
{\cal A}_2(\bar a, \bar c, \zeta) 
&=& 
(p-1) \bar a g(\bar a, \bar c, \zeta)^{p-2} + \bar c, 
\\
\label{eq:recap-str-2}
\zeta &=& \bar a g(\bar a,\bar c,\zeta)^{p-1} + \bar c g(\bar a,\bar c, \zeta).
\end{eqnarray}
We first compute the derivative of ${\cal A}_2$ with respect to $\bar
a$:
$$
\frac{\partial {\cal A}_2}{\partial \bar a} 
= 
(p-1) g(\bar a,\bar c,\zeta)^{p-2} + (p-1)(p-2) \bar a \, g(\bar a,\bar
c,\zeta)^{p-3} \frac{\partial g}{\partial \bar a}.
$$
Using~\eqref{eq:recap-str-2} to compute $\dis \frac{\partial g}{\partial
  \bar a}$, we obtain that
$$
\left( \bar c+ (p-1) \bar a g^{p-2} \right) 
\frac{\partial {\cal A}_2}{\partial \bar a} 
= 
(p-1) \bar c g^{p-2} 
+ \bar a (p-1) g^{2p-4},
$$
and since $p >1$, $\bar a > 0$ and $g>0$, we deduce that $\dis
\frac{\partial {\cal A}_2}{\partial \bar a} \geq 0$.

We next compute the derivative of ${\cal A}_2$ with respect to $\bar
c$. Using again~\eqref{eq:recap-str-2} to compute $\dis \frac{\partial
  g}{\partial \bar c}$, we obtain that
$$
\left( \bar c+ (p-1)\bar a g^{p-2} \right) 
\frac{\partial {\cal A}_2}{\partial \bar c} 
= 
\bar c - (p-1)(p-3) \bar a g^{p-2}.
$$
Recall that $\bar a > 0$, $\bar c \geq 0$, $p >1$ and $g>0$. We have
assumed that $p \leq 3$, and therefore deduce from the above relation that 
$\dis \frac{\partial {\cal A}_2}{\partial \bar c} \geq 0$. 
The structure assumption~\eqref{eq:structure-der22} hence holds in that
case. 

\begin{remark}
The argument above also shows that the case
$$
W(y,\omega,\xi) = a(y, \omega) \frac{|\xi|^p}{p},
$$
along with assumption~\eqref{eq:random-form1}, falls in our framework,
for any $p \geq 2$.

It is likely that other settings, such as
$$
W(y,\omega,\xi) = 
\left( a(y, \omega) + c(y,\omega) \right) \frac{|\xi|^p}{p} + 
c(y, \omega) \frac{|\xi|^2}{2},
$$
along with assumption~\eqref{eq:random-form1}, where $a_k$ and $c_k$ are
all independent random variables, also fall in our framework. We will
not pursue in this direction here. 
\end{remark}

\section{Numerical results}
\label{sec:num}

Our numerical experiments are presented in Section~\ref{sec:overview},
and discussed in details in the subsequent sections. In
Section~\ref{sec:newton}, we first discuss the algorithm we used to
solve the variational problem~\eqref{eq:correc-random_N} that defines
the apparent homogenized energy density.

\subsection{Newton algorithm to solve the truncated corrector problem}
\label{sec:newton}

As mentioned above, the corrector problem~\eqref{eq:correc-random_N} is
a convex minimization problem, which has been well studied in the
literature (see e.g.~\cite{barrett,chow,glowinski,letallec}). We explain
here how we proceed in practice to solve this problem, assuming 
that $W$ is not only strictly convex, but actually $\alpha$-convex
(i.e. satisfies~\eqref{eq:hyp_alpha_convex}).

To simplify our exposition, we use
the notation of the $Q$-periodic case, where the corrector problem
is~\eqref{eq:champ-homog_per}. We introduce some basis functions $\left\{
  \varphi_i \right\}_{i \in I}$ (e.g. finite element functions) 
where $\varphi_i \in W^{1,p}_\#(Q)$, and
the finite dimensional space
$V_h = \text{Span} \left\{ \varphi_i , i \in I \right\}$. Consider the
functional
$$
J(w) =  J \left( \left\{ w_i \right\}_{i \in I} \right) = 
\int_Q W(y, \xi+\nabla w(y)) \, dy
$$
defined on $V_h$, with 
$$
w(y) = \sum_{i \in I} w_i \ \varphi_i(y),
$$
and the variational problem
\begin{equation}
\label{eq:num-min-pb}
\inf_{v_h \in V_h} J(v_h).
\end{equation}
This problem has a unique solution (denoted $w_h \in V_h$) up to the
addition of a constant. The quantity $\nabla w_h$ is well-defined, and
is the finite-dimensional approximation of $\nabla w$, where $w$ is the
solution to~\eqref{eq:champ-homog_per}.

In practice, problem~\eqref{eq:num-min-pb} is solved using a Newton
algorithm. We see that 
$$
\frac{\partial J}{\partial w_j}(w) = D_w (\varphi_j)
\quad \text{and} \quad
\frac{\partial^2 J}{\partial w_j \partial w_k}(w)
= H_w (\varphi_j, \varphi_k)
$$ 
where
$$
D_w (\varphi)
= \int_Q \nabla \varphi(y) \cdot \partial_\xi W (y, \xi + \nabla w(y)) \, dy
$$
and
$$
H_w (\varphi, \psi) = \int_Q \left( \nabla \varphi(y) \right)^T 
\partial^2_\xi W(y, \xi + \nabla w(y)) \, \nabla \psi(y) \, dy.
$$
The Newton algorithm consists in defining $w_h^{m+1} \in V_h$ from
$w_h^m \in V_h$ by the following linear elliptic problem: find
$w_h^{m+1} \in V_h$ such that 
$$
\forall \theta \in V_h, \quad
H_{w_h^m} (w_h^{m+1} - w_h^m, \theta) = - D_{w_h^m} (\theta).
$$
Again, $w_h^{m+1}$ is uniquely defined up to the addition of a
constant. 

The finite-dimensional problem~\eqref{eq:num-min-pb} is
$\alpha$-convex, and $W$ is smooth with respect to $\xi$: the Newton
algorithm hence locally converges (quadratically), and $\dis \lim_{m
  \rightarrow \infty} \nabla w_h^m = \nabla w_h$. 

In practice, we consider a sequence ${\cal T}_{h}$ of meshes on $Q$, and
set $\dis V_h = \PP^1_h(Q) = \left\{ v_h \in C(Q) \text{ s.t. } \forall T \in
  {\cal T}_h, v_h \text{ is affine on $T$} \right\}$. 
By classical finite element results, we know that 
$\dis \lim_{h \to 0} \| \nabla w_h - \nabla w \|_{L^p(Q)} = 0$ (see
e.g.~\cite{barrett} and also~\cite{thomee,abdulle_vilmart}). 

\subsection{Overview of numerical results}
\label{sec:overview}

We have considered three test-cases of the
form~\eqref{eq:test-case}--\eqref{eq:random-form1}, namely
\begin{gather*}
W(y,\omega,\xi) = 
a(y, \omega) \frac{|\xi|^p}{p} + c(y, \omega) \frac{|\xi|^2}{2}
\\
\text{with }
a(y,\omega) = \sum_{k \in \mathbb{Z}^d} \mathbf{1}_{Q + k}(y)
a_k(\omega)
\ \ \text{and} \ \ 
c(y,\omega) = \sum_{k \in \mathbb{Z}^d} \mathbf{1}_{Q + k}(y) c_k(\omega),
\end{gather*}
with $p=4$, in dimension $d=2$. The
random variables $a_k$ follow a Bernoulli distribution:
$\PP(a_k = \alpha) = \PP(a_k = \beta) = 1/2$, with 
$\alpha=3$ and $\beta=23$. The value of the field $c$ is chosen as
follows:
\begin{itemize}
\item Test Case 1: in this first test case, $c(y,\omega) = 0$. The
  problem is thus strictly convex but not $\alpha$-convex. In
  addition, the energy density is positively homogeneous of degree $p$,
  hence Remarks~\ref{rem:mon-hom} and~\ref{rem:mon-hom2} apply.
\item Test Case 2: the second test case corresponds to $c(y,\omega) =
  1$. The problem 
  is then $\alpha$-convex, and highly oscillatory only in its non-harmonic
  component. 
\item Test Case 3: for the third test case, we work with $c(y,\omega)$ chosen
  according to~\eqref{eq:random-form1}, where 
$\PP(c_k = \gamma) = \PP(c_k = \delta) = 1/2$, with 
$\gamma=1$ and $\delta=3$. The problem is thus highly oscillatory both
in its non-harmonic and its harmonic components. 
\end{itemize}
We take the meshsize $h=0.2$.
The Newton algorithm is initialized with the solution $w_0$ to
$$
-\hbox{\rm div}\left[ (a(y,\omega) + c(y,\omega)) (\xi + \nabla w_0) \right] = 0
\quad \text{in $Q_N$}, 
\quad
\text{$w_0$ is $Q_N$-periodic},
$$
and the iterations stop when
$\dis \frac{\| w_h^{n+1} - w_h^n \|_{W^{1,p}}}{\| w_h^n \|_{W^{1,p}}} 
\leq {\tt tol}$. 
If {\tt tol} is chosen too large, then~\eqref{eq:num-min-pb} is
inaccurately solved, and the variance reduction is not very good. For our
numerical tests, we set ${\tt tol} = 10^{-5}$: the discrete
problem~\eqref{eq:num-min-pb} is accurately solved, while only a
limited number of iterations (in practice, around 5 iterations) are
needed. 

For the numerical tests, we adopt the convention that $Q_N =
(-N/2,N/2)^2$. For each $Q_N$, the standard Monte Carlo results have
been obtained using $2M=100$ realizations (from which we build the
empirical estimator~\eqref{eq:estim1}). For the antithetic variable
approach, we have also solved $2M$ corrector problems, from which we
build the
empirical estimator~\eqref{eq:estim2}. Therefore, in all what follows,
we compare the accuracy of the Monte Carlo approach (MC) and the Antithetic
Variable approach (AV) at {\em equal computational cost}. 

\subsection{Test Case 1}

In this test case, the energy density is positively homogeneous. We
therefore know, from Proposition~\ref{prop:strong} and
Remark~\ref{rem:mon-hom2}, that our approach yields estimations of the
expectation of $W^\star_N(\omega,\xi)$, 
$\xi \cdot \partial_\xi W^\star_N(\omega,\xi)$ and
$\xi^T \partial^2_\xi W^\star_N(\omega,\xi) \xi$ with a smaller variance
than the standard Monte Carlo approach. 
Our aim here is to {\em quantify} the efficiency gain. Note also that we
have {\em not} taken into account, in our implementation, the fact that
$W^\star_N(\omega,\xi)$, $\xi \cdot \partial_\xi W^\star_N(\omega,\xi)$
and $\xi^T \partial^2_\xi W^\star_N(\omega,\xi) \xi$ are here
proportional to one another.

\medskip

To begin with, we show on Figure~\ref{per-homog-wrt-dof} the estimation by
empirical means (along with a 95 \% confidence interval) of several
homogenized quantities (the energy density, its derivatives with respect
to each component of $\xi$, \dots). We observe that the variance
of all quantities decreases when the size of $Q_N$ increases, and
that confidence intervals obtained with the antithetic variable approach
are smaller than those obtained with a standard Monte Carlo approach, for
an equal computational cost. 

\begin{figure}[htpb!]
\begin{center}
\includegraphics[width=5.35cm]{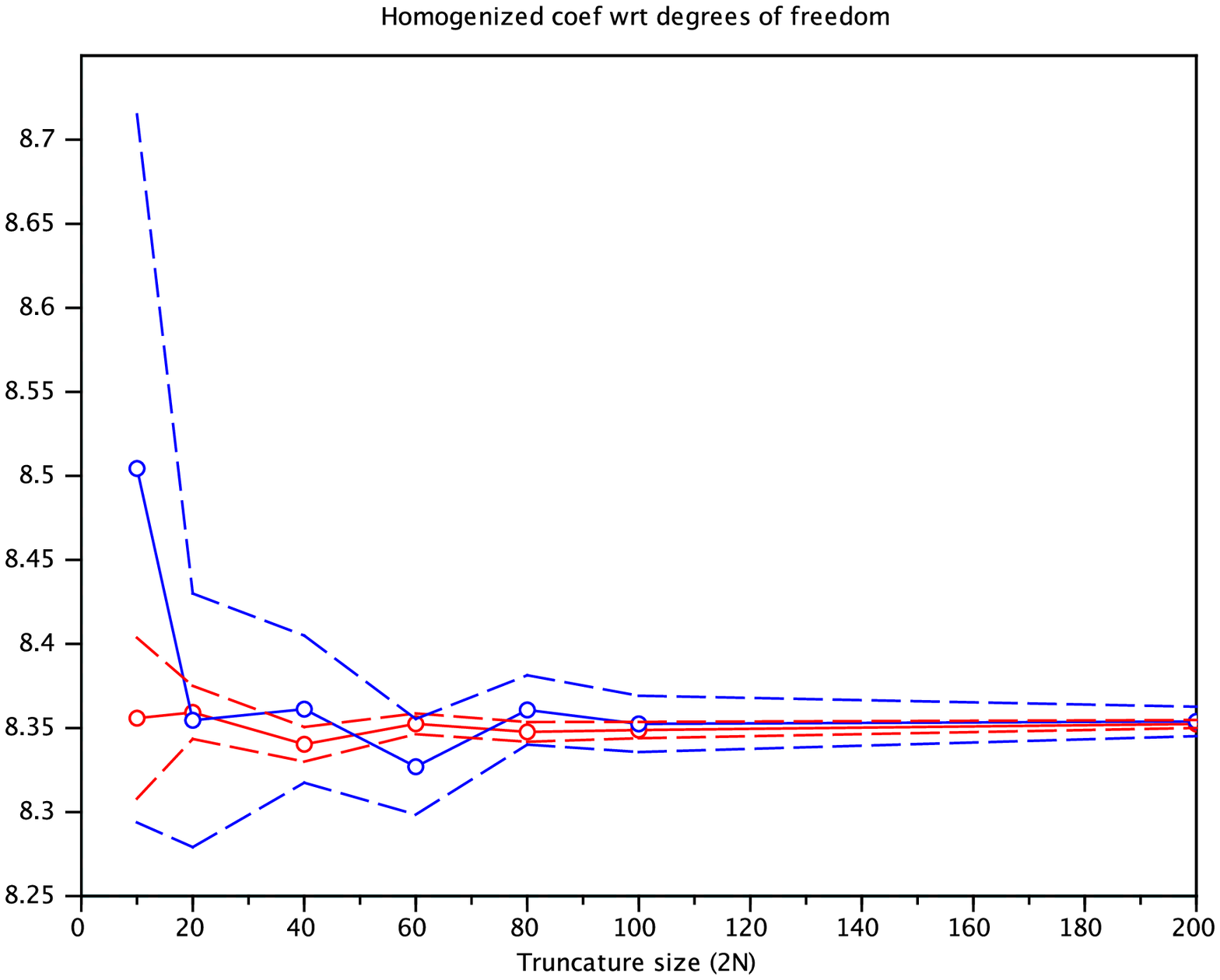}
\quad
\includegraphics[width=5.35cm]{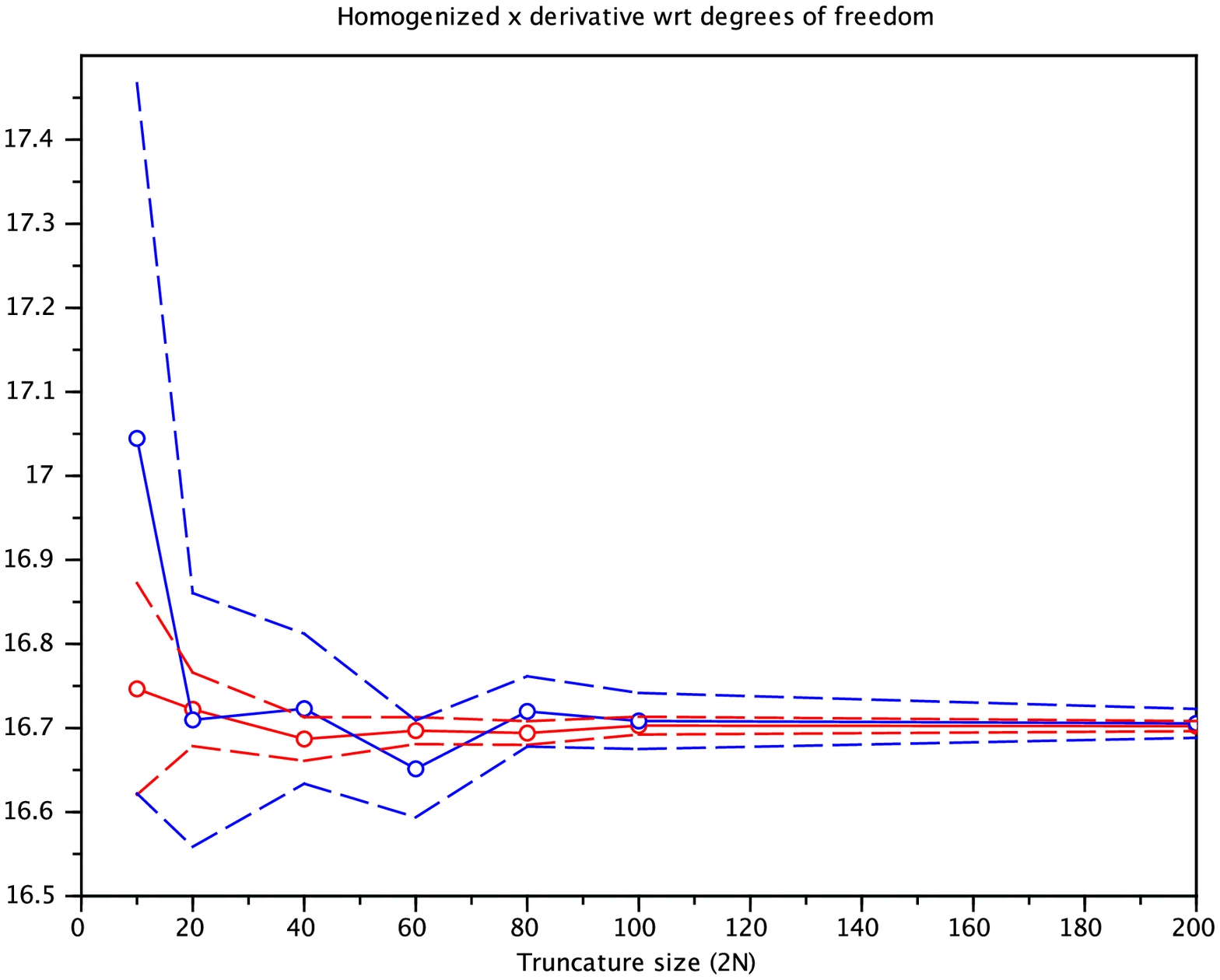}
\includegraphics[width=5.35cm]{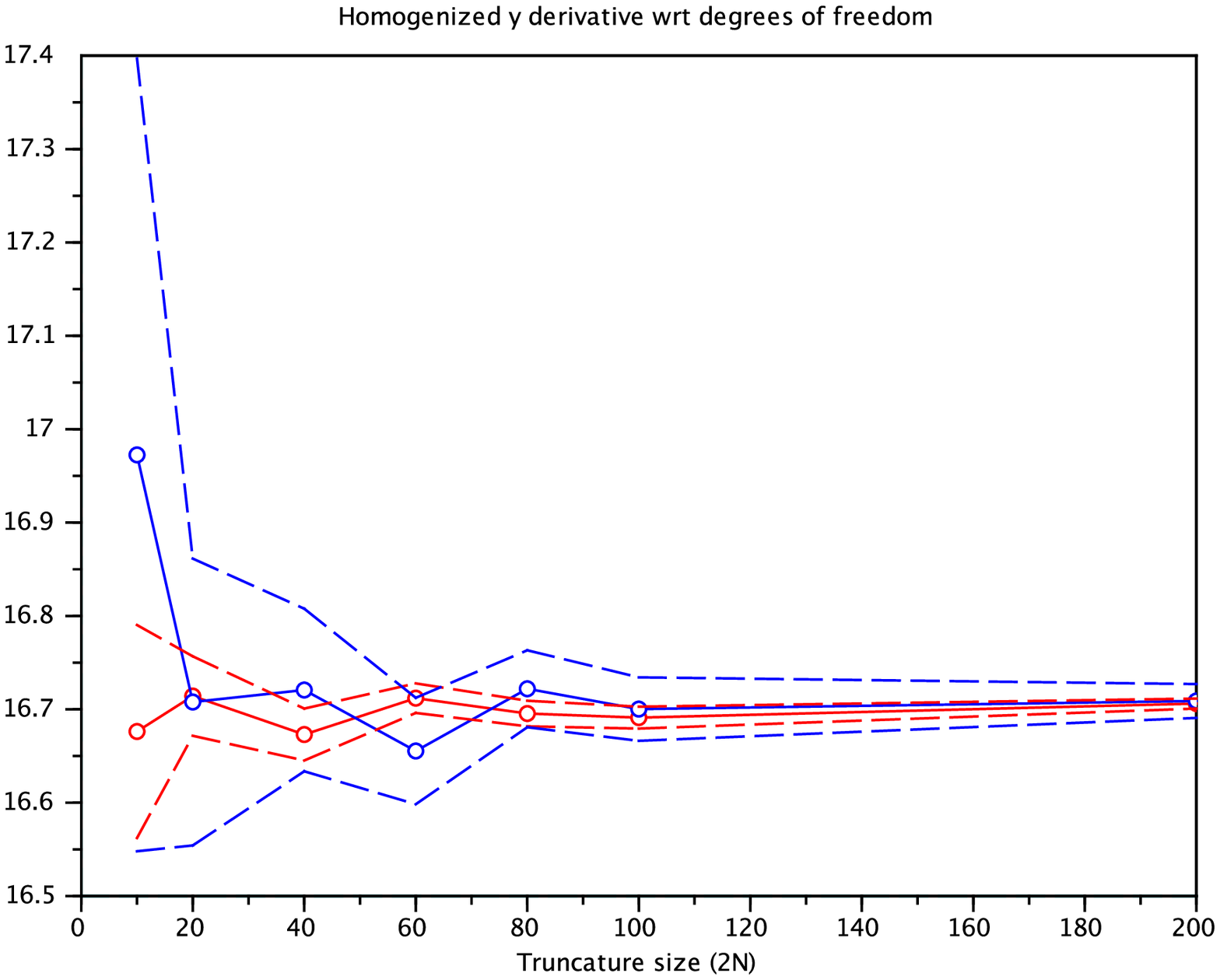}
\quad
\includegraphics[width=5.35cm]{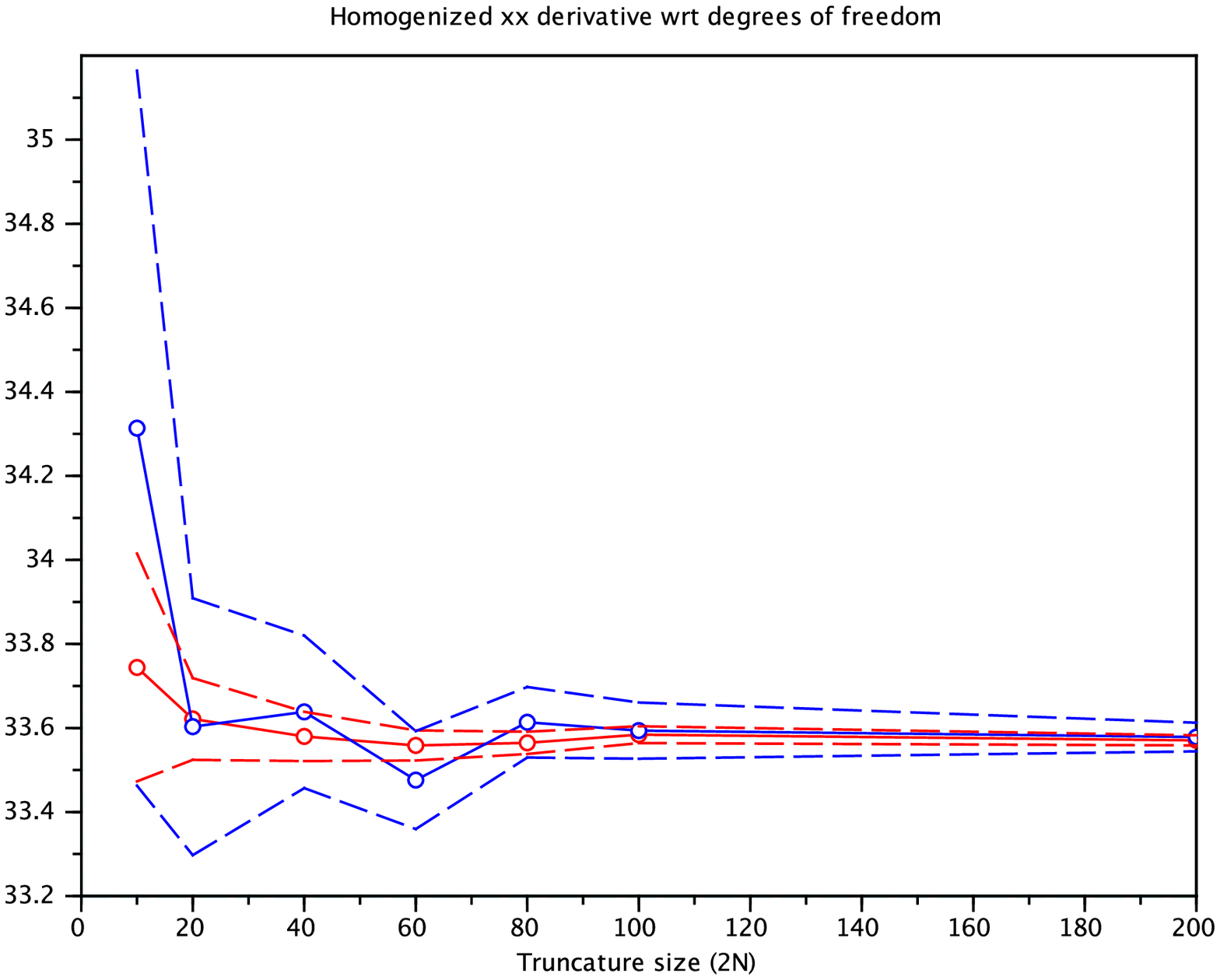}
\includegraphics[width=5.35cm]{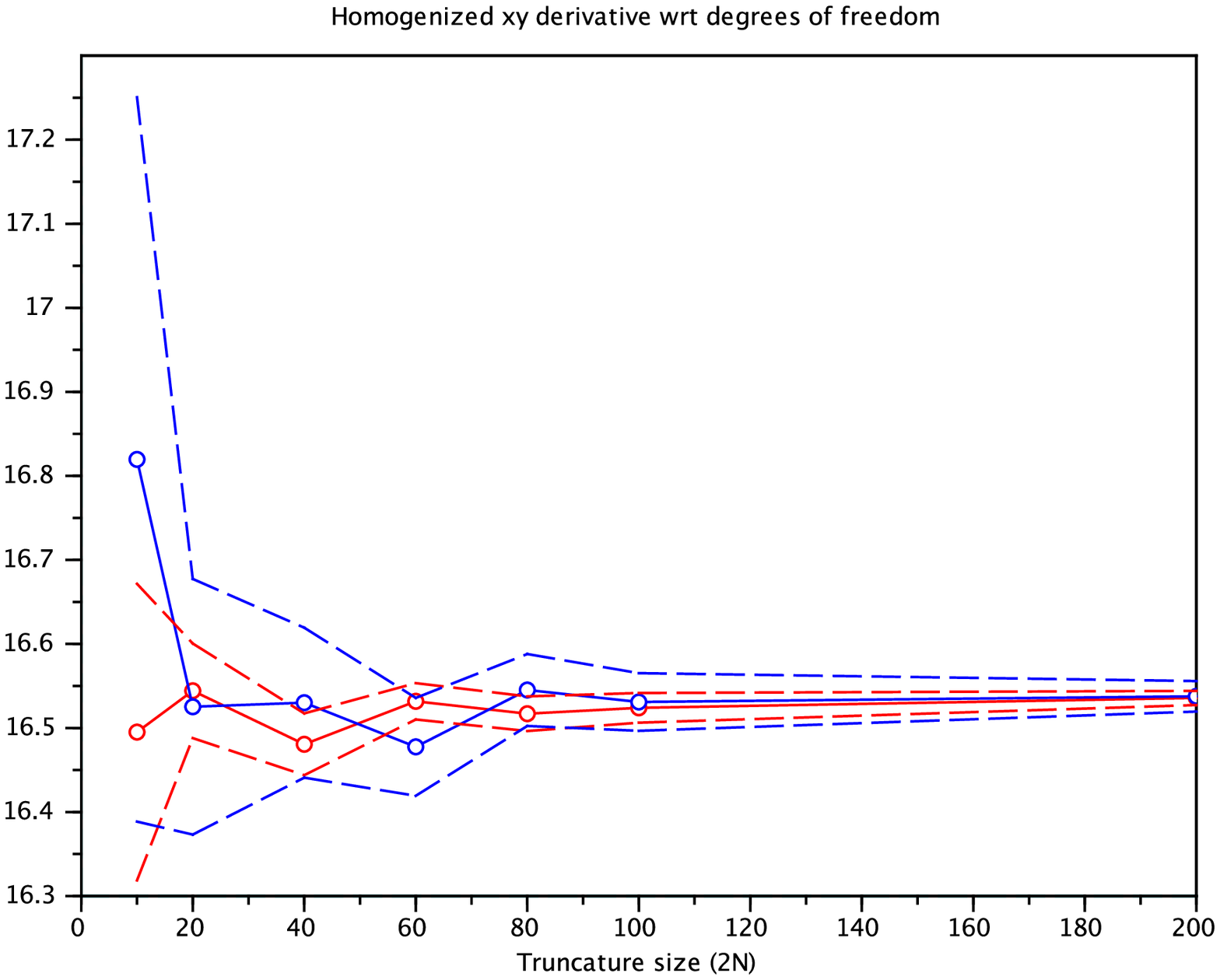}
\quad
\includegraphics[width=5.35cm]{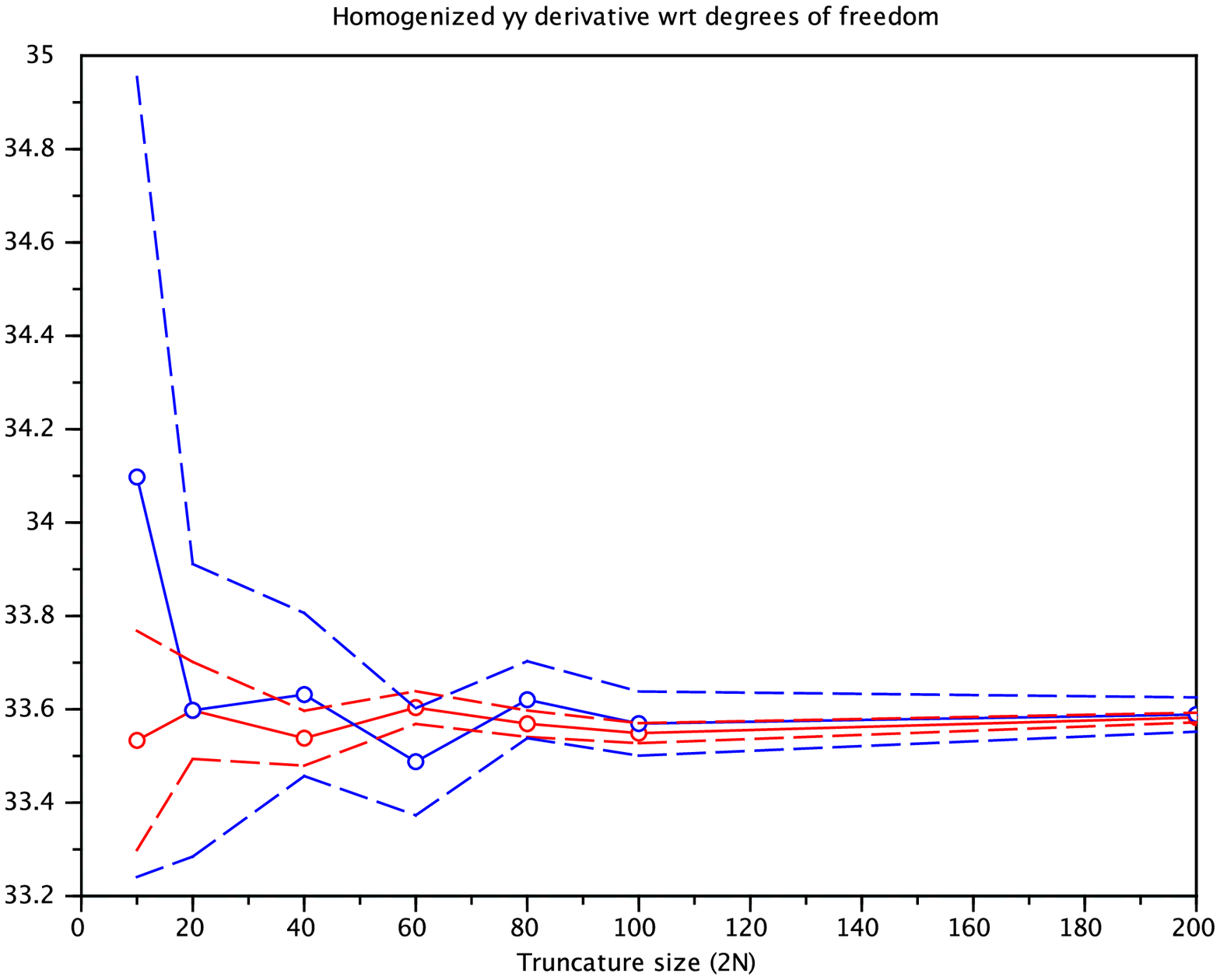}
\includegraphics[width=5.35cm]{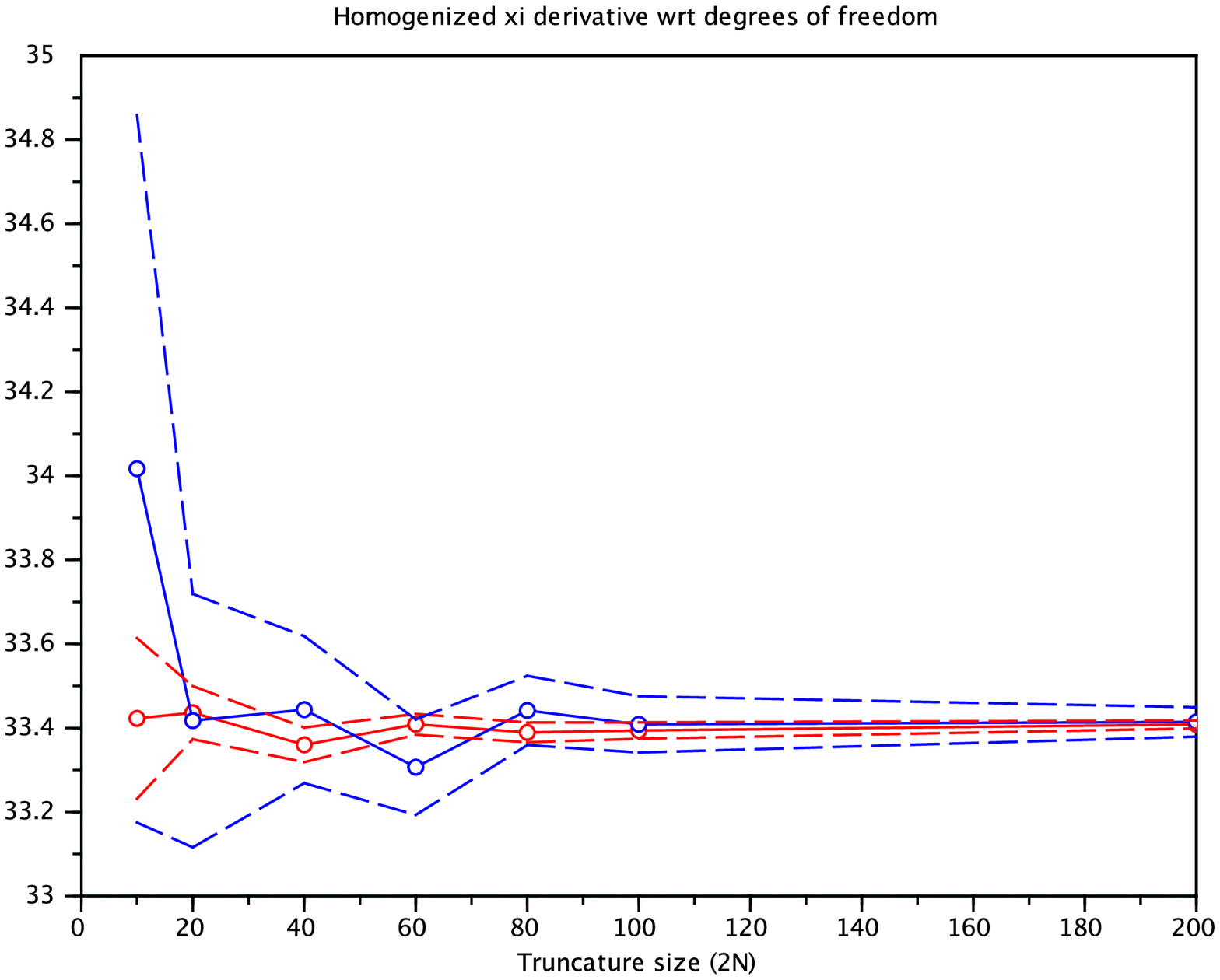}
\quad
\includegraphics[width=5.35cm]{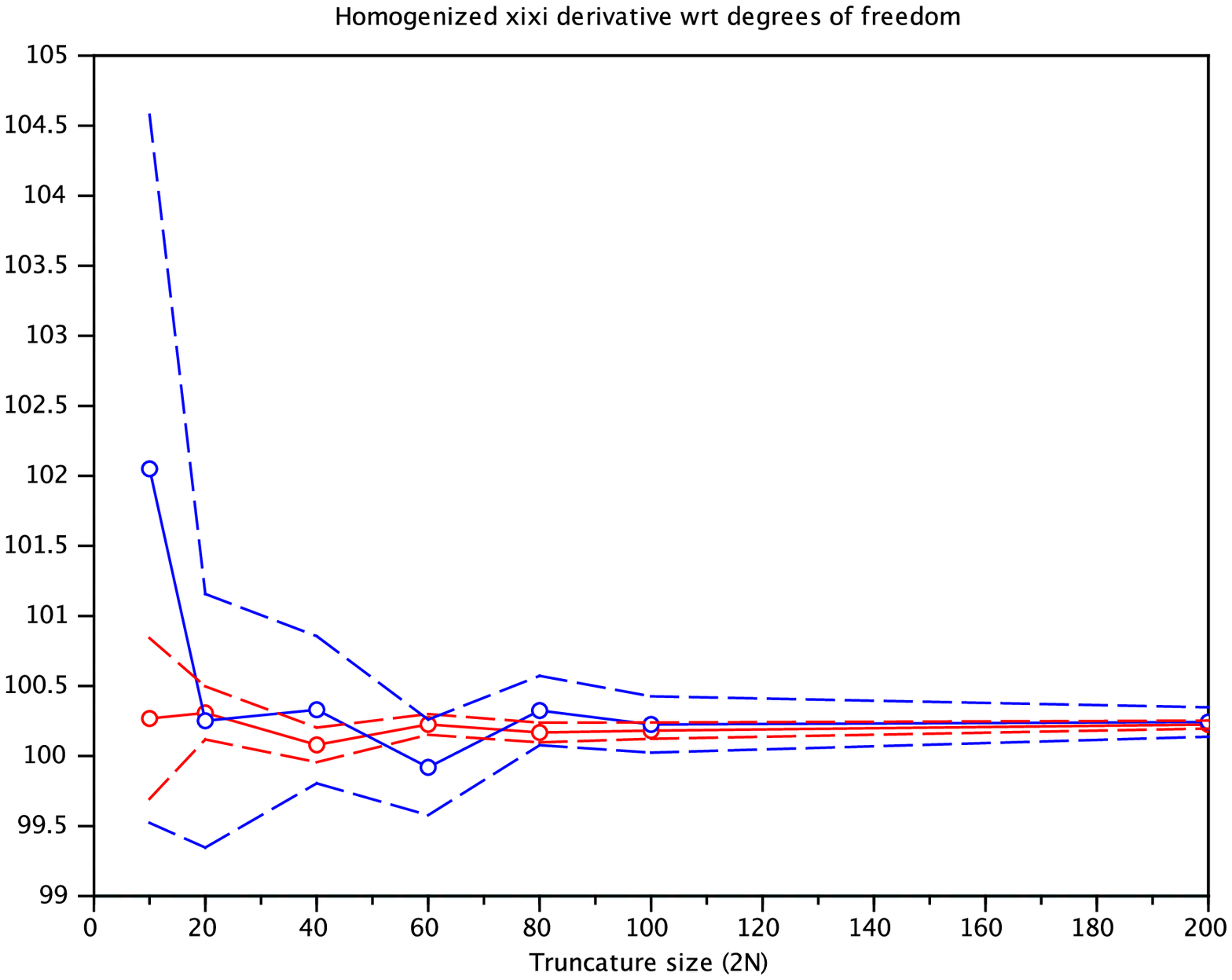}
\caption{Test-Case 1: Homogenized quantities as a function of $N$, for
  the vector $\xi=(1,1)^T$ (Blue: Monte Carlo
  results; Red: Antithetic variable approach; Dashed lines: 95\%
  confidence interval, equating the cost of the two approaches). 
From top left to bottom right: $\EE \left[ W^\star_N(\cdot,\xi) \right]$,
$\EE \left[ \partial_{\xi_1} W^\star_N(\cdot,\xi) \right]$,
$\EE \left[ \partial_{\xi_2} W^\star_N(\cdot,\xi) \right]$,
$\EE \left[ \partial_{\xi_1 \xi_1} W^\star_N(\cdot,\xi) \right]$,
$\EE \left[ \partial_{\xi_1 \xi_2} W^\star_N(\cdot,\xi) \right]$,
$\EE \left[ \partial_{\xi_2 \xi_2} W^\star_N(\cdot,\xi) \right]$,
$\EE \left[ \xi \cdot \partial_\xi W^\star_N(\cdot,\xi) \right]$ and
$\EE \left[ \xi^T \partial^2_\xi W^\star_N(\cdot,\xi) \xi \right]$.
\label{per-homog-wrt-dof}}
\end{center}
\end{figure}

We now turn to a more quantitative analysis of the
variance. Figure~\ref{per-plap-var-wrt-dof} shows the variances
\begin{equation}
\label{eq:def_variances}
V_{\rm MC} = \frac{1}{2} \Var \left[ W^\star_N(\cdot,\xi) \right]
\quad \text{and} \quad
V_{\rm AV} = \Var \left[ \widetilde{W}^\star_N(\cdot,\xi) \right]
\end{equation}
as a function of $N$ (note the factor $1/2$ in the definition of $V_{\rm
  MC}$, consistent with~\eqref{eq:resu0_1},~\eqref{eq:estim1}
and~\eqref{eq:estim2}). We observe that the variance of
any of our quantities of interest (obtained either with the Monte
Carlo approach or the Antithetic Variable approach) decreases at the
rate $1/|Q_N|$ as $N$ increases (as expected if one could use the
Central Limit Theorem). We also observe that the variance obtained with
our approach is systematically smaller than the Monte Carlo
variance, in the sense that $V_{\rm AV} \leq V_{\rm MC}$.

\begin{figure}[h!]
\begin{center}
\includegraphics[width=5.1cm]{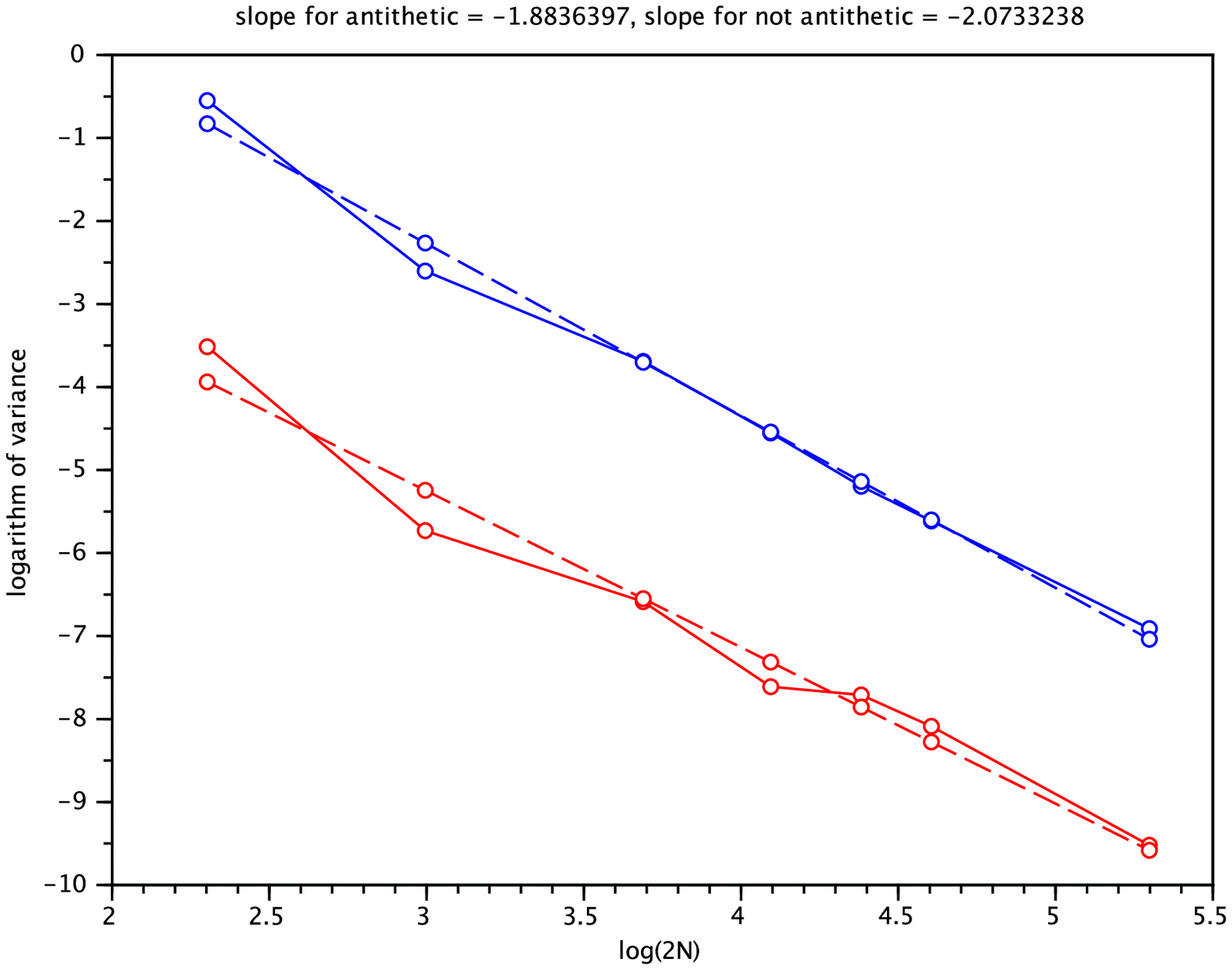}
\quad
\includegraphics[width=5.1cm]{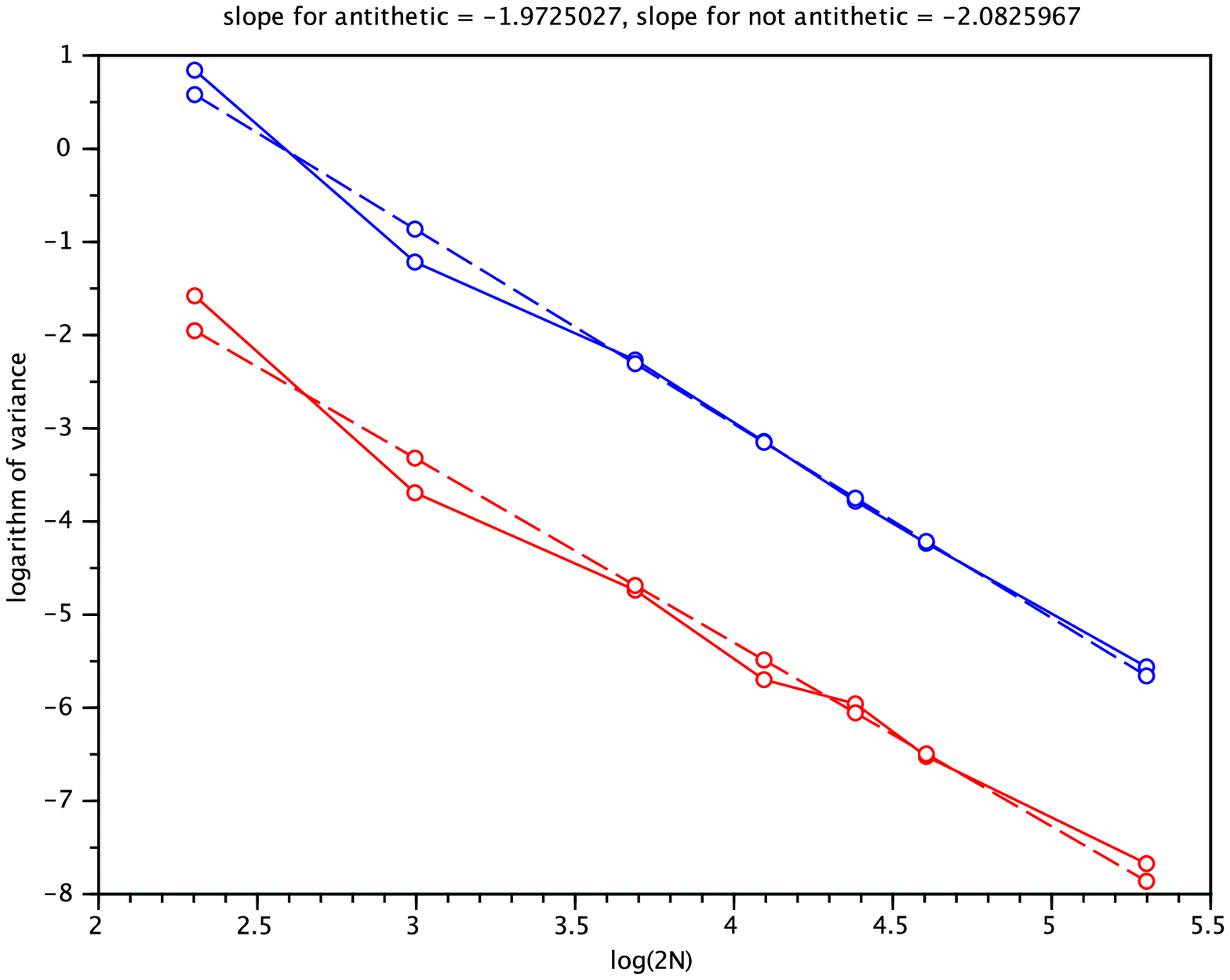}
\includegraphics[width=5.1cm]{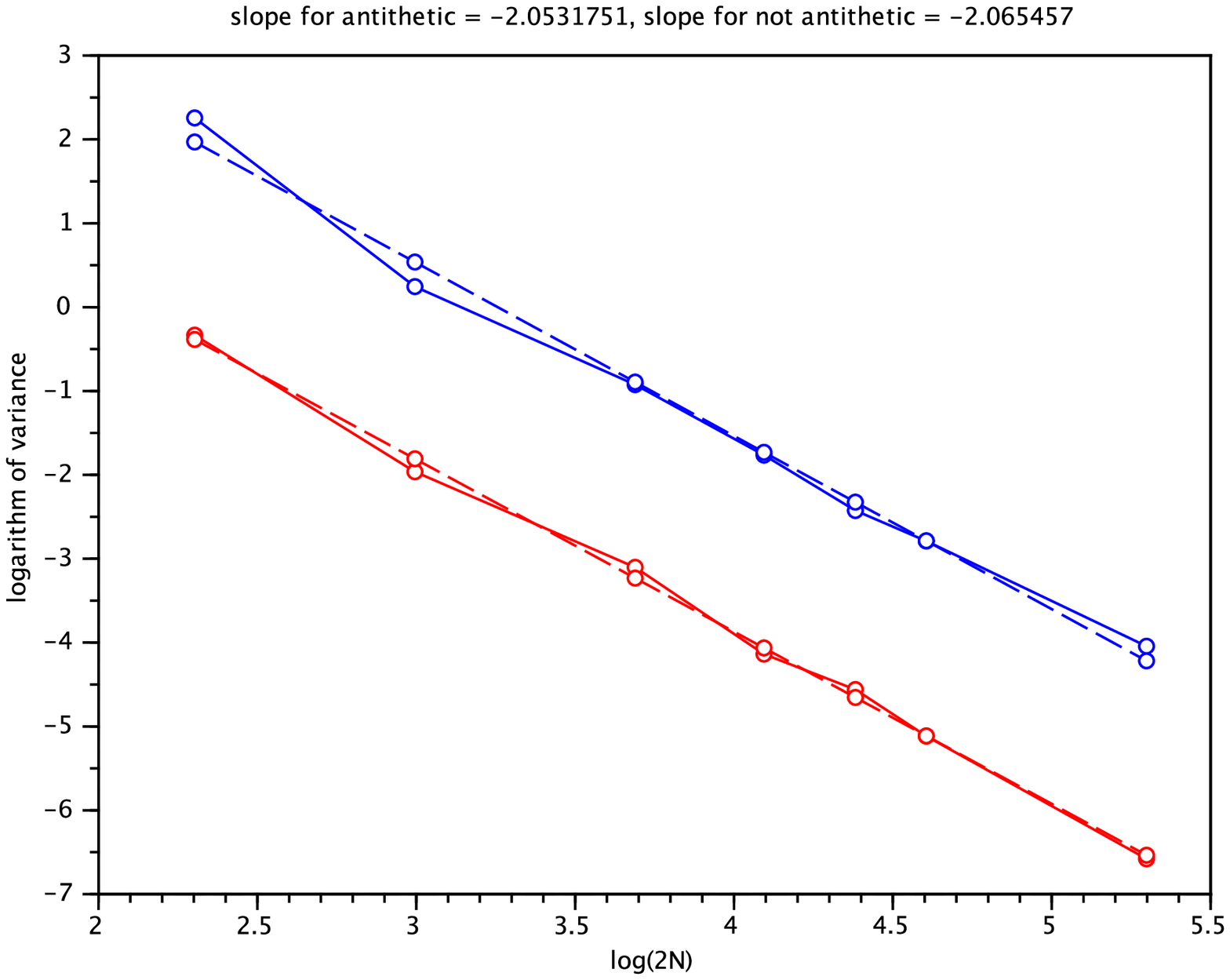}
\quad
\includegraphics[width=5.1cm]{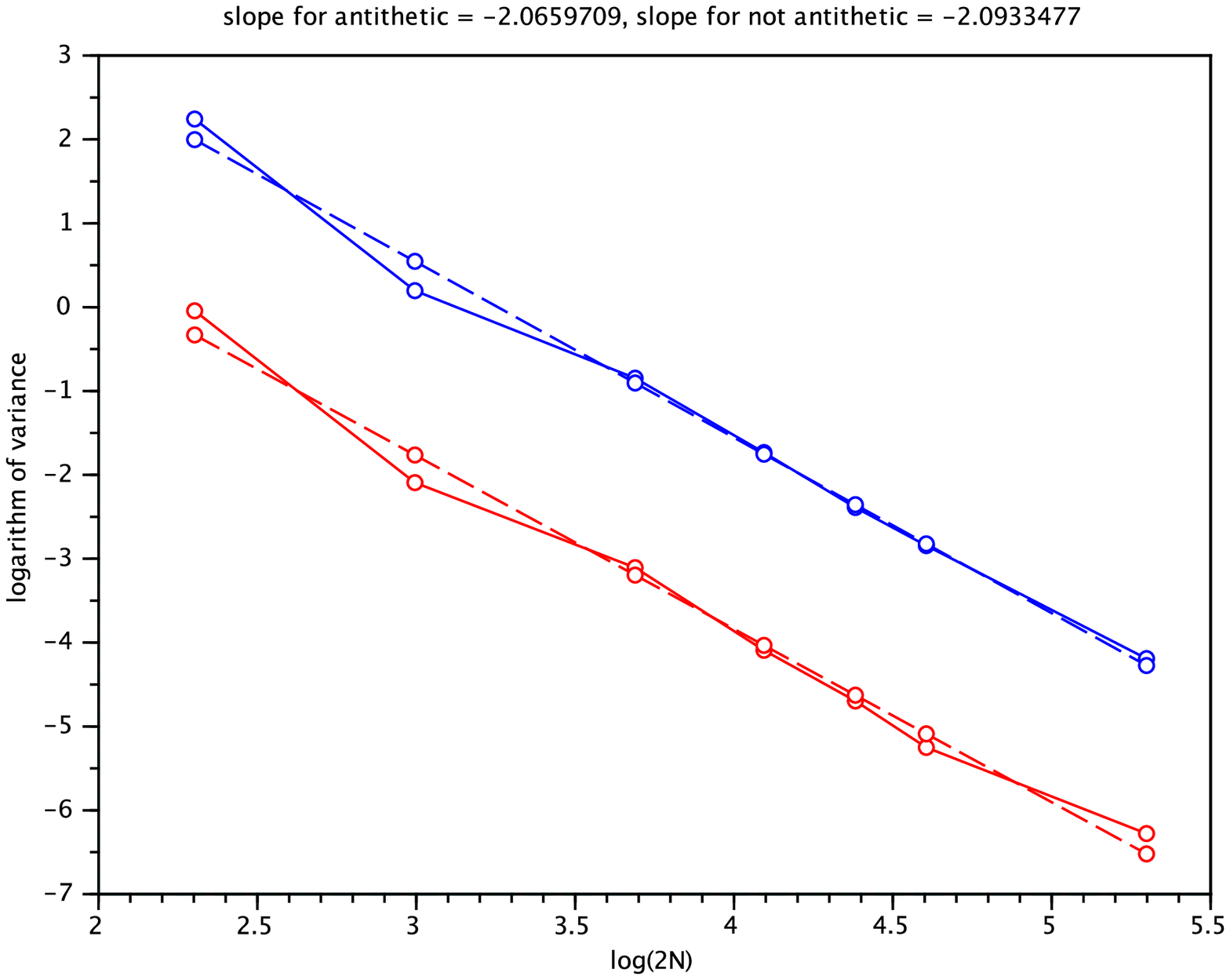}
\includegraphics[width=5.1cm]{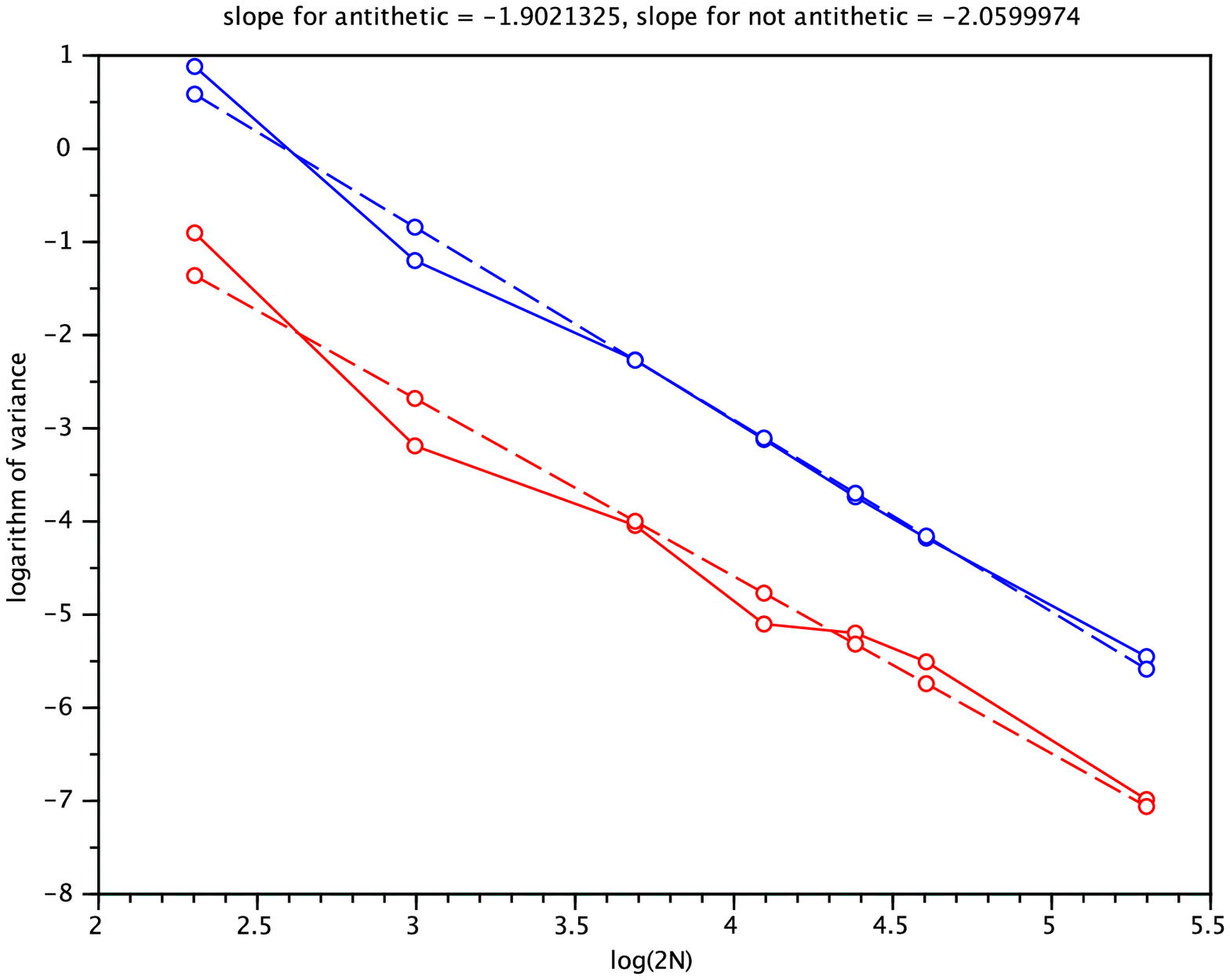}
\quad
\includegraphics[width=5.1cm]{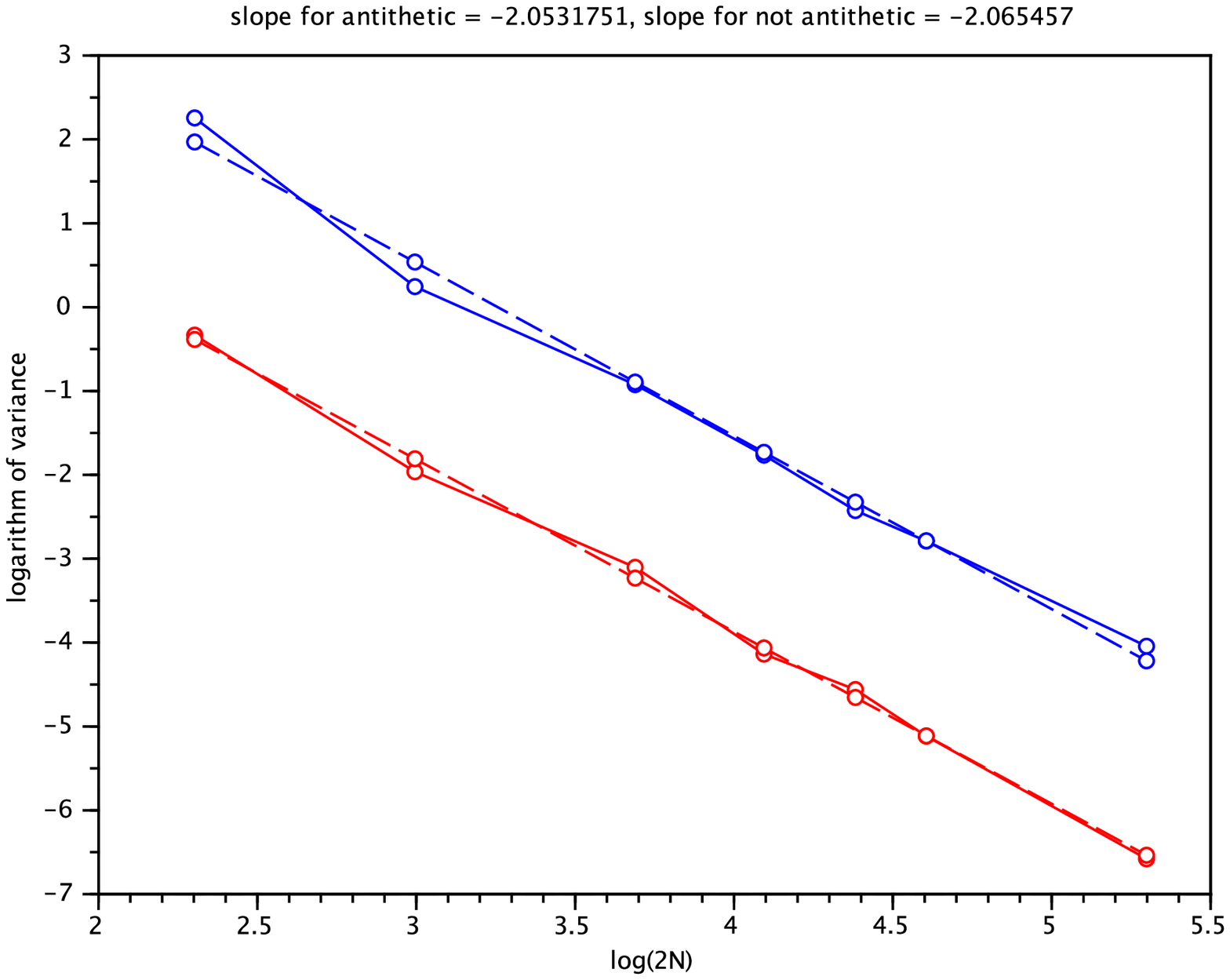}
\includegraphics[width=5.1cm]{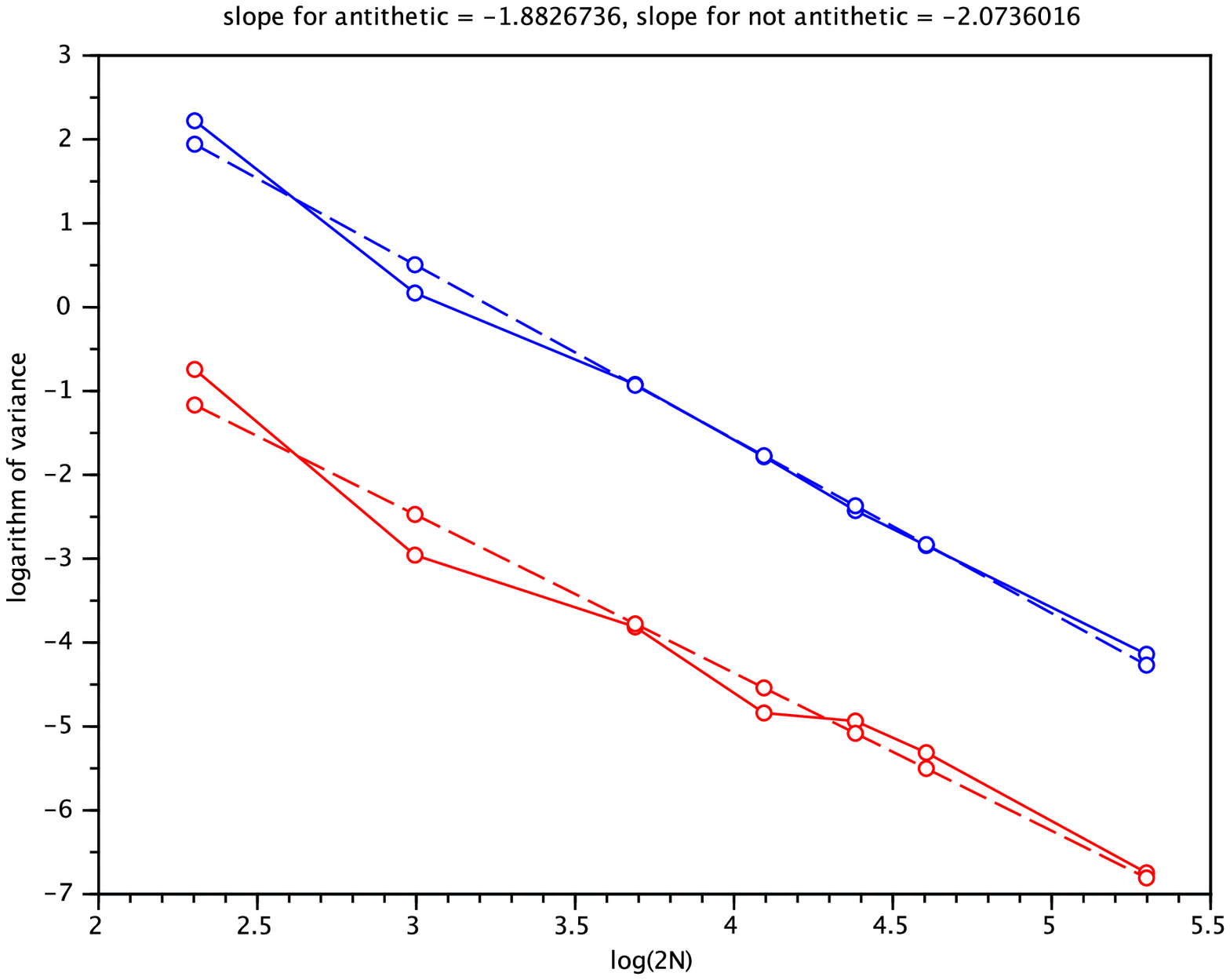}
\quad
\includegraphics[width=5.1cm]{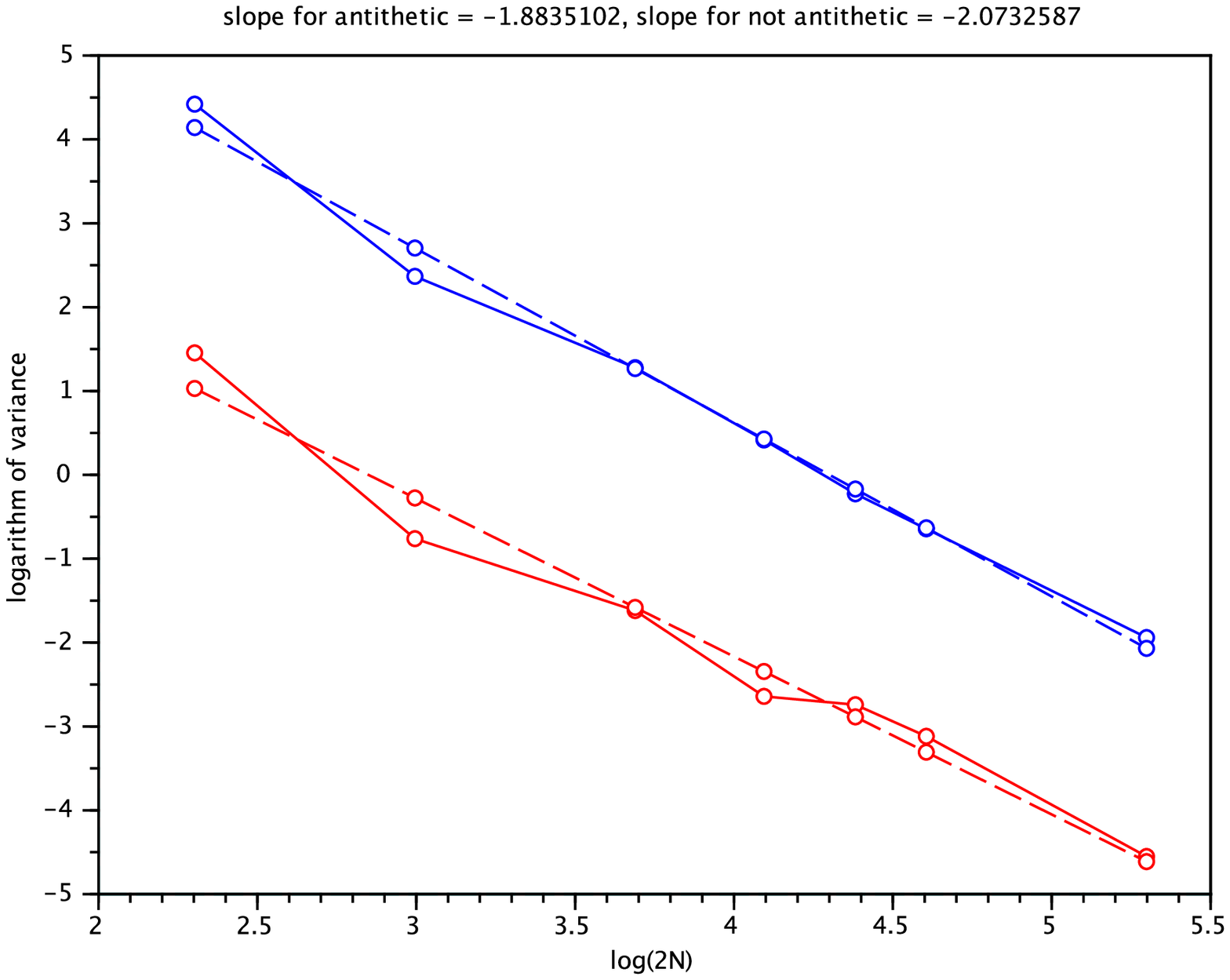}
\caption{Test-Case 1: Variances~\eqref{eq:def_variances} of the same 
quantities of interest as on Figure~\ref{per-homog-wrt-dof}, 
as a function of $N$ (Blue: Monte Carlo approach; 
  Red: Antithetic Variable approach).
\label{per-plap-var-wrt-dof}}
\end{center}
\end{figure}

We next report on Table~\ref{tab:p-lap-var-red} the variance reduction
ratio
\begin{equation}
\label{eq:def_R}
R 
= 
\frac{V_{\rm MC}}{V_{\rm AV}}
=
\frac{
\Var \left[ W^\star_N(\cdot,\xi) \right]
}
{2 \Var \left[ \widetilde{W}^\star_N(\cdot,\xi) \right]},
\end{equation}
which measures the gain in computational cost at equal accuracy, or the
square of the accuracy gain at equal computational cost. Although this
ratio somewhat varies with $N$, we observe
that it is of the order of 10 for all quantities of
interest, except for $\partial_{\xi_1 \xi_2} W^\star_N$, for which it is
always larger than 4.
In particular, even if $N$ is not large (because we cannot afford
to work on a large domain $Q_N$), we still observe variance reduction. 

\begin{table}[!h]
\begin{center}
\begin{tabular}{|c|c|c|c|c|c|c|c|c|}
$2N$ & $W^\star_N$ & $\partial_{\xi_1} W^\star_N$ & $\partial_{\xi_2}
W^\star_N$ & $\partial_{\xi_1 \xi_1} W^\star_N$ & $\partial_{\xi_1 \xi_2}
W^\star_N$ & $\partial_{\xi_2 \xi_2} W^\star_N$ &
$\xi \cdot \partial_\xi W^\star_N$ & 
$\xi^T \partial^2_\xi W^\star_N \xi$ \\
\hline
10 & 19.41 & 11.26 & 13.86 & 9.846 & 5.966 & 13.34 & 19.39 & 19.41  
\\
20 & 22.82 & 11.89 & 13.03 & 9.865 & 7.306 & 9.096 & 22.77 & 22.83
\\
40 & 18.08 & 11.82 & 9.816 & 9.576 & 5.904 & 8.831 & 18.03 & 18.11  
\\
60 & 21.26 & 12.89 & 12.98 & 10.57 & 7.247 & 10.73 & 21.24 & 21.28   
\\
80 & 12.36 & 8.798 & 9.050 & 10.05 & 4.316 & 8.454 & 12.31 & 12.37
\\
100 & 11.88 & 9.856 & 8.412 & 11.10 & 3.775 & 10.24 & 11.82 & 11.88 
\\
200 & 13.60 & 8.261 & 11.52 & 8.057 & 4.636 & 12.62 & 13.54 & 13.61
\\
\end{tabular}
\caption{Test-Case 1: Variance reduction ratios~\eqref{eq:def_R}.\label{tab:p-lap-var-red}}
\end{center}
\end{table}

\begin{remark}
Similar variance reduction ratios are obtained in the case when the
corrector problem is supplemented with homogeneous Dirichlet boundary
conditions on the boundary on $Q_N$, rather than periodic boundary
conditions as used here following~\eqref{eq:correc-random_N} (results
not shown). 
\end{remark}

\subsection{Test Case 2}

We now consider a test-case for which the energy density is not
positively homogeneous. From our results of
Section~\ref{sec:main_results}, we know that our
approach yields variance reduction for the estimation of $\EE \left[
  W^\star_N(\cdot,\xi) \right]$. Our aim here is two-fold: we first quantify
the efficiency gain, and we next verify (and this will indeed be the
case) that we also obtain a gain in efficiency for quantities of
interest (such as the first or second derivatives of $W^\star_N(\omega,\xi)$
with respect to $\xi$) for which we do not have theoretical results in
the two-dimensional case. 

We show on Figure~\ref{per-non-linear-log-var-wrt-dof} the
variances~\eqref{eq:def_variances} of the same
quantities of interest as previously (obtained either with the Monte
Carlo approach or the Antithetic Variable approach). As for the previous
test-case, we observe that all variances decrease at the
rate $1/|Q_N|$ as $N$ increases. In addition, we observe that the
variance obtained with our approach is systematically smaller than
the Monte Carlo variance, in the sense that $V_{\rm AV} \leq V_{\rm MC}$.

\begin{figure}[h!]
\begin{center}
\includegraphics[width=5.3cm]{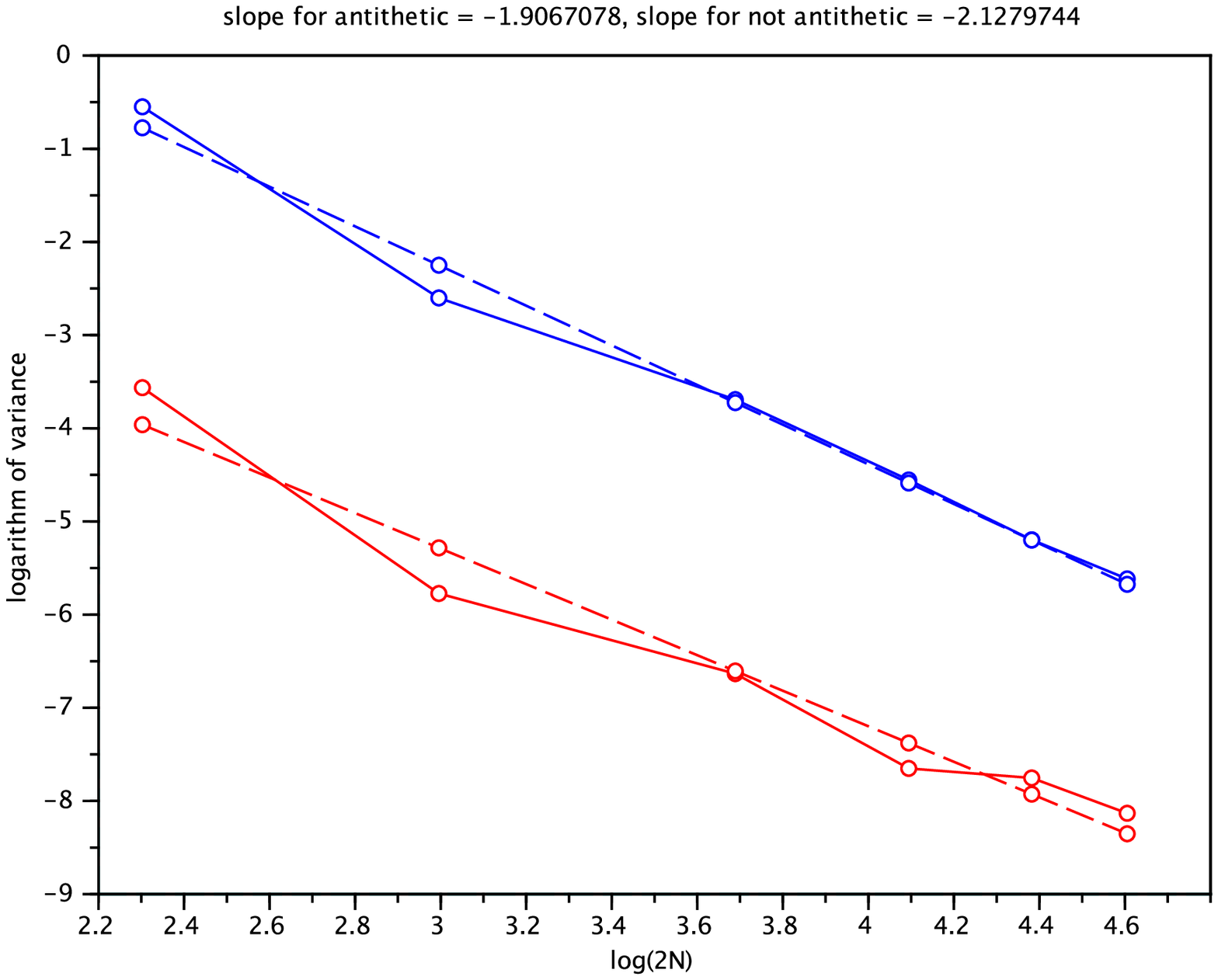}
\quad
\includegraphics[width=5.3cm]{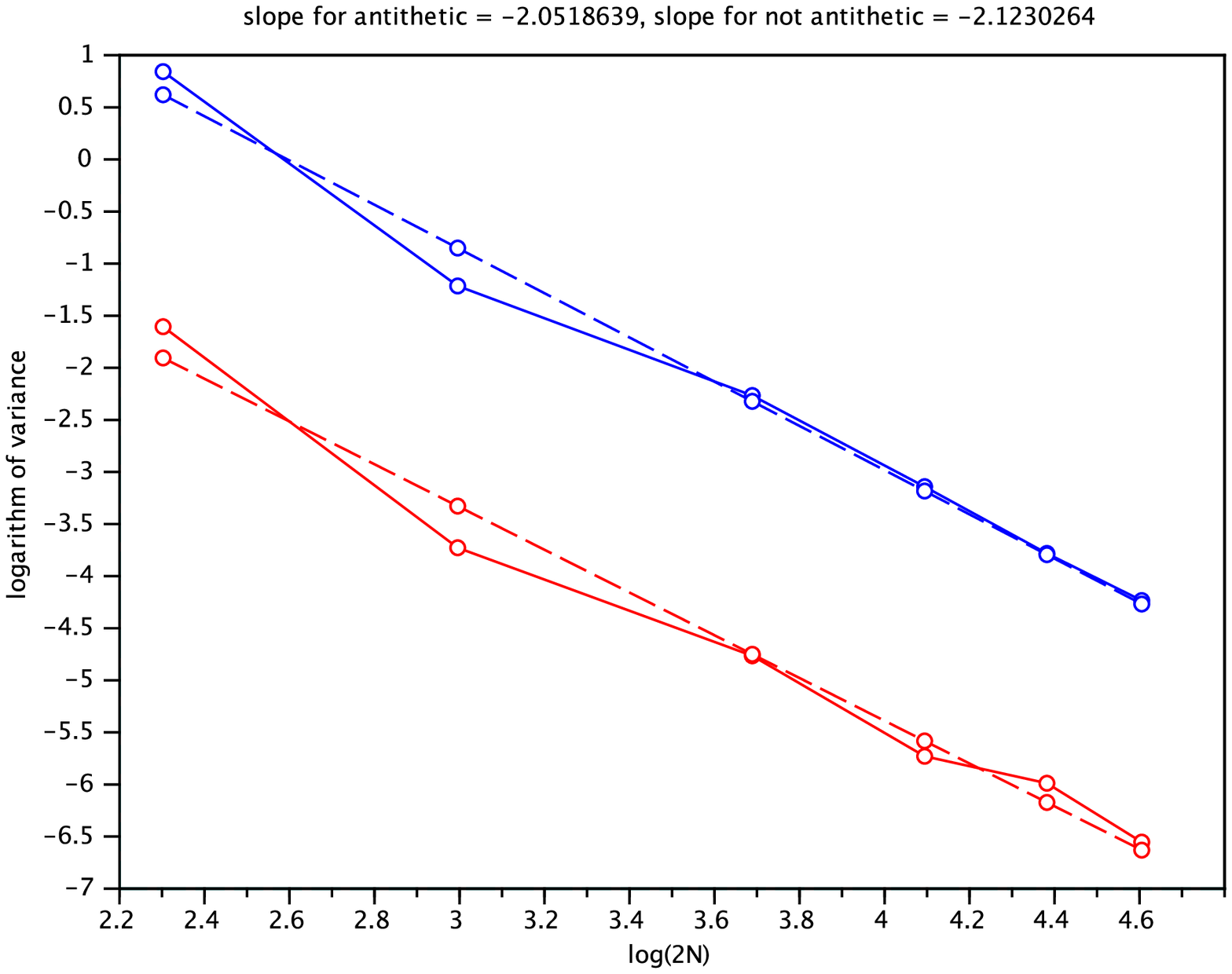}
\includegraphics[width=5.3cm]{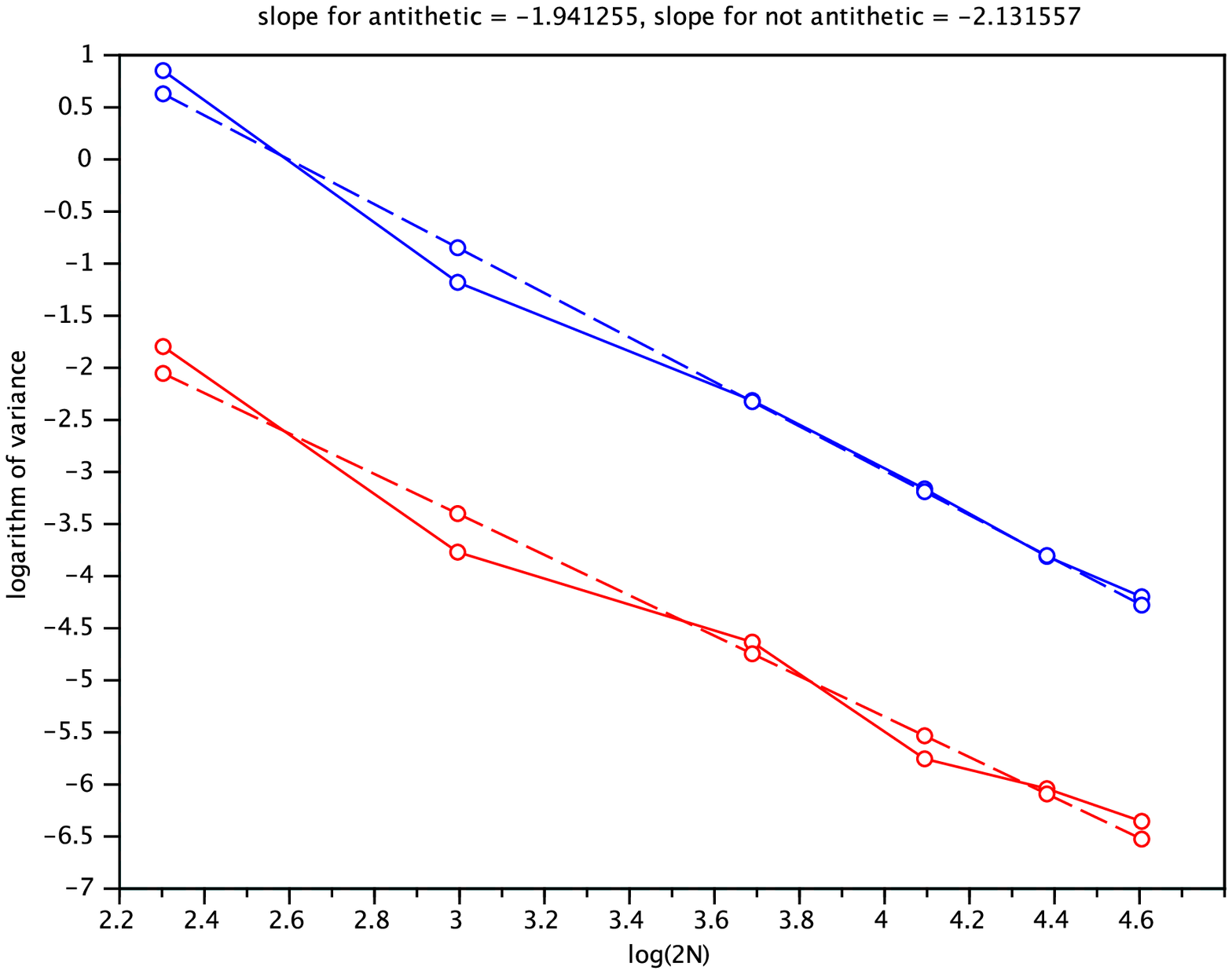}
\quad
\includegraphics[width=5.3cm]{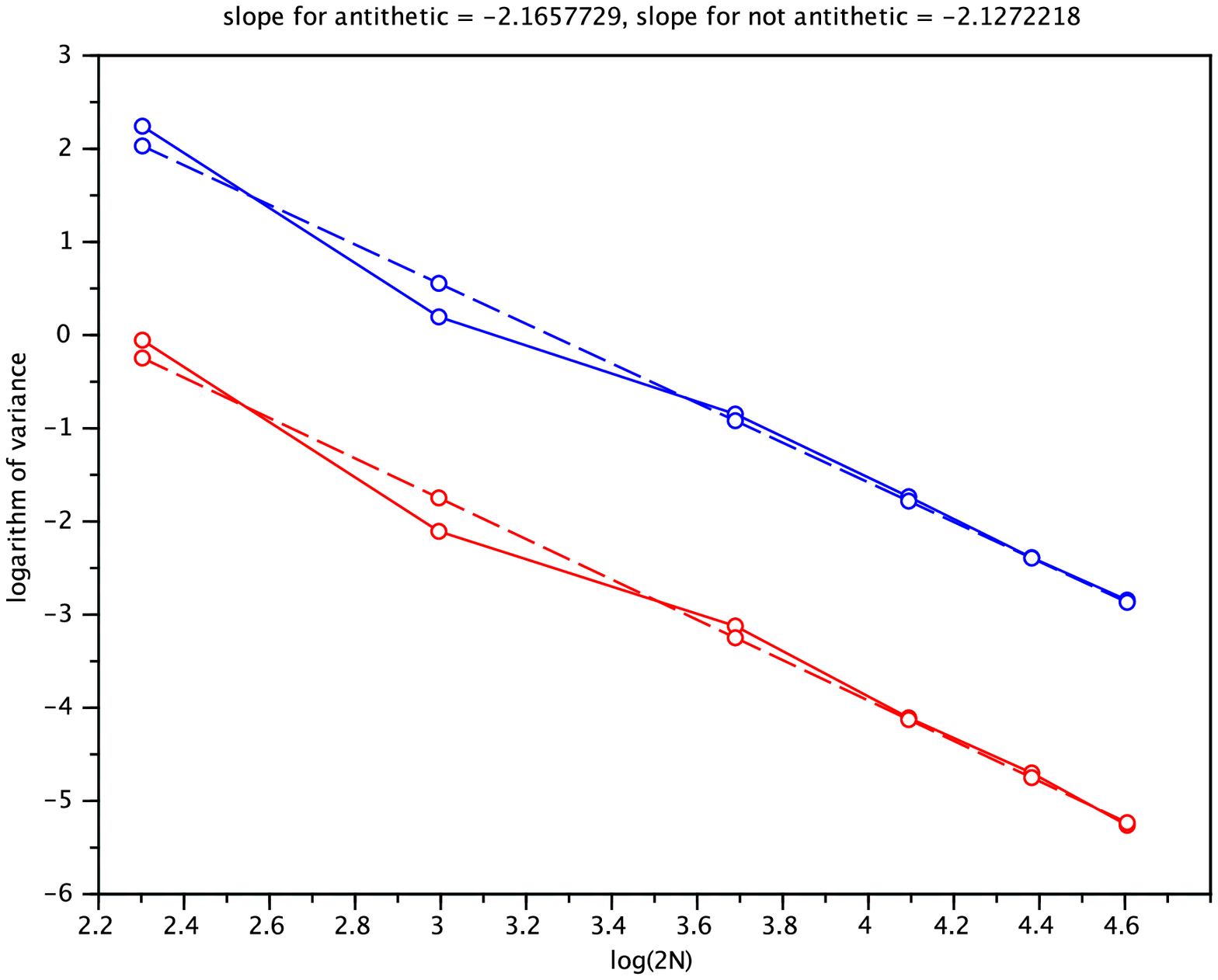}
\includegraphics[width=5.3cm]{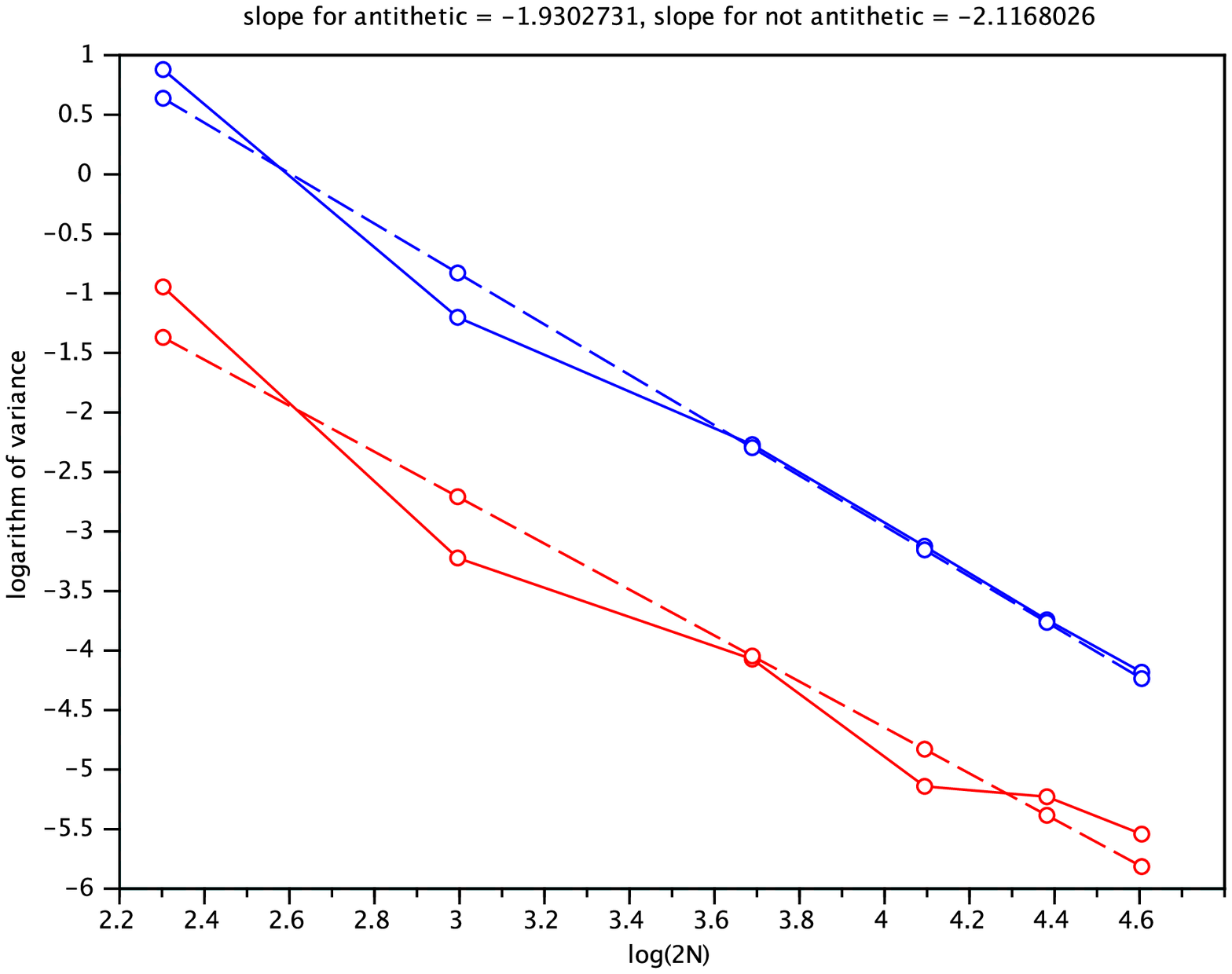}
\quad
\includegraphics[width=5.3cm]{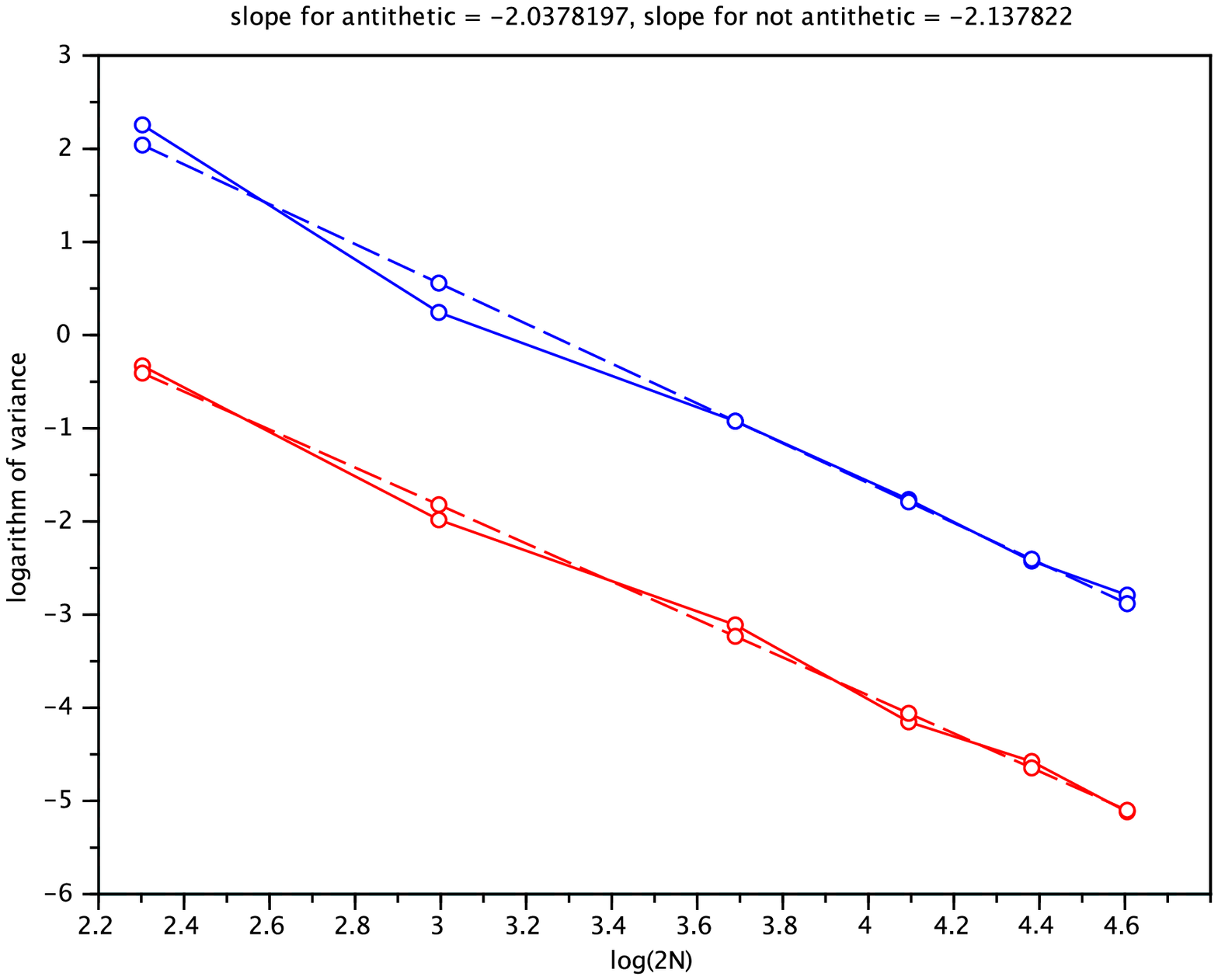}
\includegraphics[width=5.3cm]{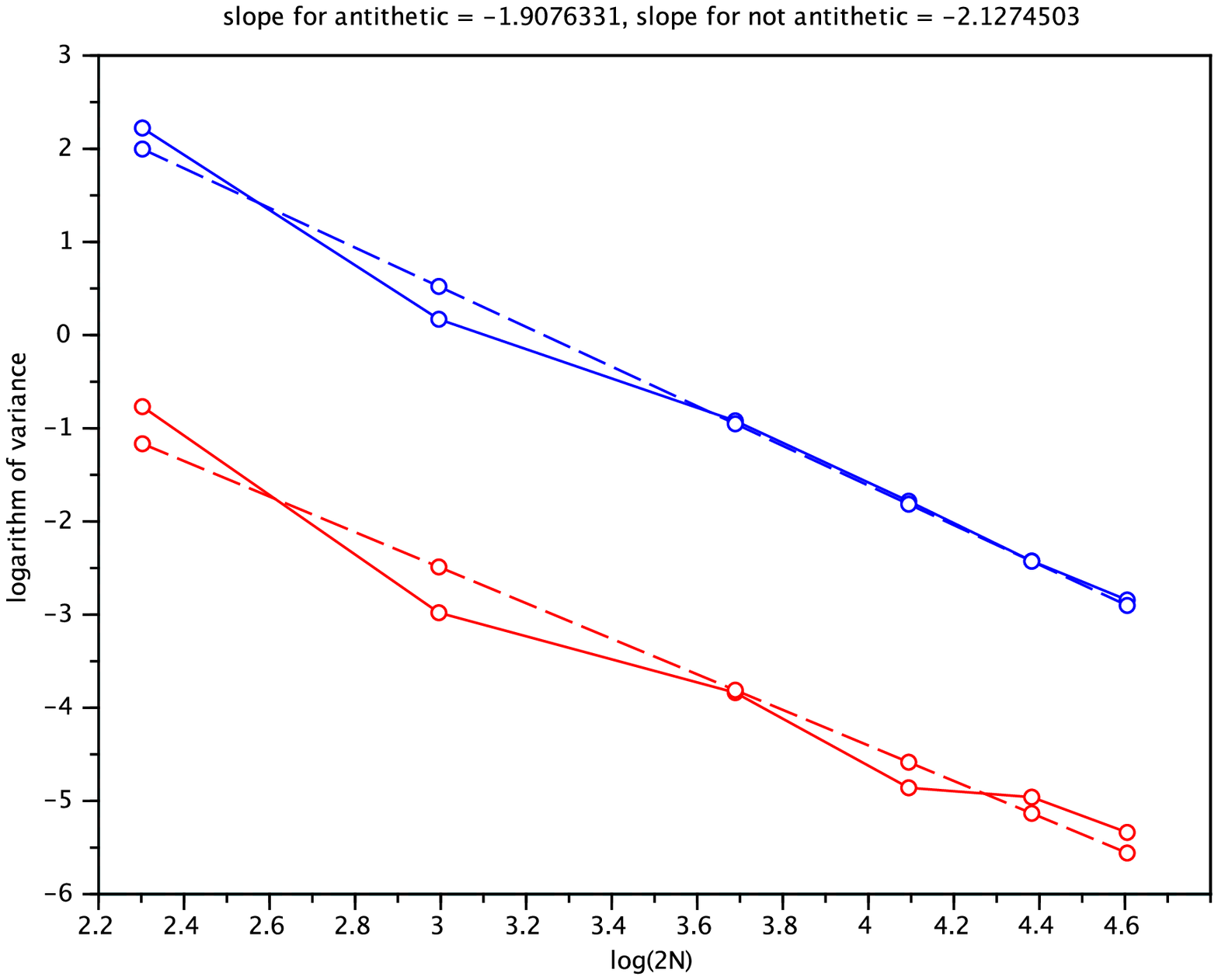}
\quad
\includegraphics[width=5.3cm]{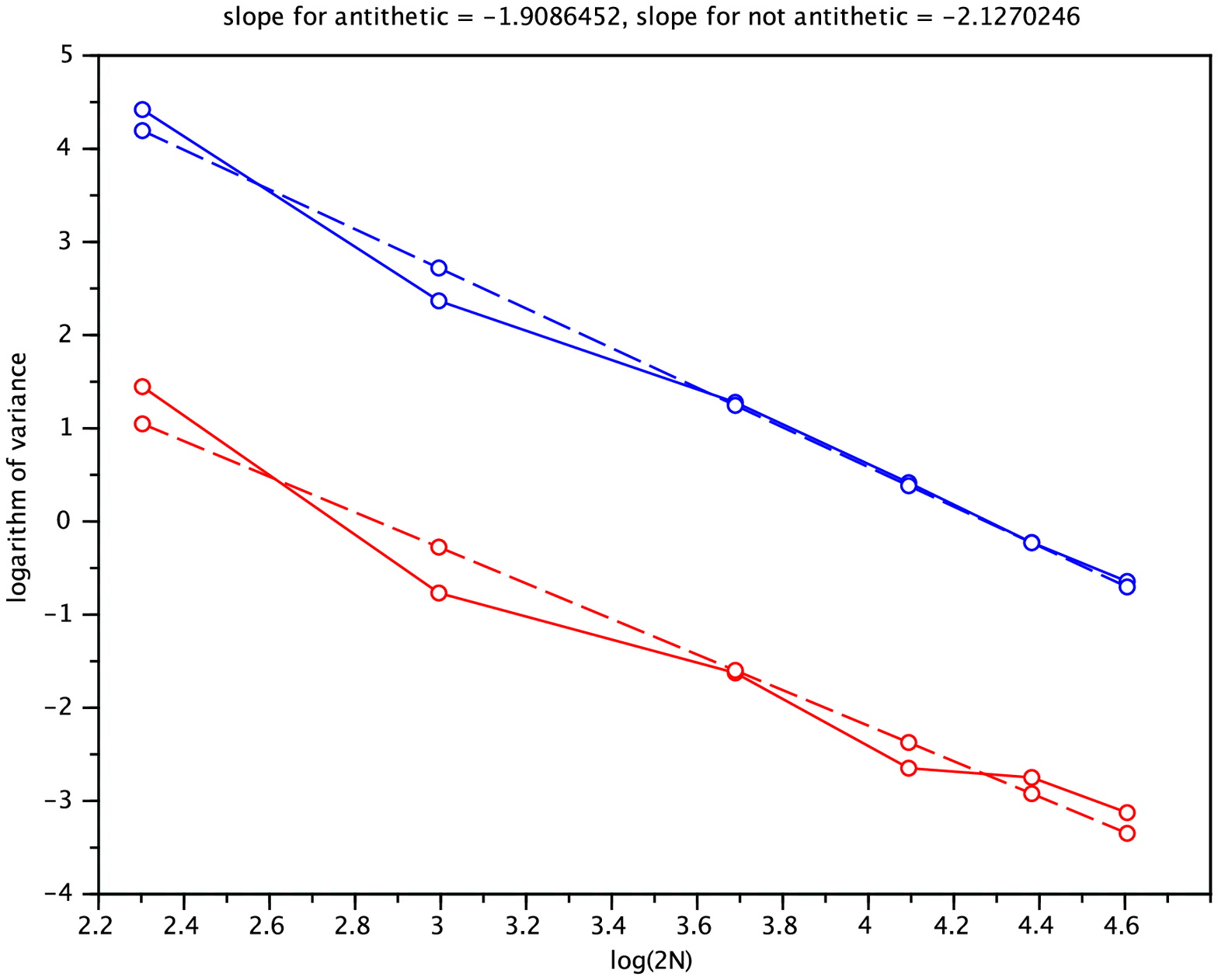}
\caption{Test-Case 2: Variances~\eqref{eq:def_variances} as a function
  of $N$ (Blue: Monte Carlo approach; Red: Antithetic Variable
  approach). The quantities of interest are the same as on
  Figure~\ref{per-homog-wrt-dof}.
\label{per-non-linear-log-var-wrt-dof}}
\end{center}
\end{figure}

On Table~\ref{tab:non-linear-var-red}, we report the variance reduction
ratios~\eqref{eq:def_R} (with the same convention as in
Table~\ref{tab:p-lap-var-red}). We observe an efficiency gain of more
than 10 for all quantities of interest, except again the cross
derivative $\partial_{\xi_1 \xi_2} W^\star_N$, for which the gain is
smaller, and of the order of 4. 

\begin{table}[!h]
\begin{center}
\begin{tabular}{|c|c|c|c|c|c|c|c|c|}
$2N$ & $W^\star_N$ & $\partial_{\xi_1} W^\star_N$ & $\partial_{\xi_2}
W^\star_N$ & $\partial_{\xi_1 \xi_1} W^\star_N$ & 
$\partial_{\xi_1 \xi_2} W^\star_N$ & $\partial_{\xi_2 \xi_2} W^\star_N$ &
$\xi \cdot \partial_\xi W^\star_N$ & $\xi^T \partial^2_\xi W^\star_N \xi$ \\
\hline
10 & 20.38 & 11.57 & 14.14 & 9.940 & 6.206 & 13.28 & 19.89 & 19.57 
\\
20 & 23.86 & 12.34 & 13.32 & 9.993 & 7.548 & 9.265 & 23.33 & 23.00 
\\
40 & 18.94 & 12.16 & 10.16 & 9.726 & 6.060 & 8.902 & 18.50 & 18.24 
\\
60 & 22.11 & 13.30 & 13.35 & 10.73 & 7.513 & 10.88 & 21.68 & 21.41 
\\
80 & 12.89 & 9.080 & 9.295 & 10.09 & 4.420 & 8.598 & 12.61 & 12.45
\\
100 & 12.37 & 10.17 & 8.635 & 11.21 & 3.896 & 10.24 & 12.12 & 11.96 
\\
\end{tabular}
\caption{Test-Case 2: Variance reduction ratios~\eqref{eq:def_R}.
\label{tab:non-linear-var-red}}
\end{center}
\end{table}

\subsection{Test Case 3}

We eventually turn to our final test-case, where both coefficients $a$
and $c$ do depend on the space variable.

We show on Figure~\ref{per-non-homog-log-var-wrt-dof} the
variances~\eqref{eq:def_variances} of our quantities of
interest. Again, we observe that they all decrease at the rate $1/|Q_N|$
as $N$ increases, and that the variance obtained with our approach is
systematically smaller than the Monte Carlo
variance.

\begin{figure}[h!]
\begin{center}
\includegraphics[width=5.3cm]{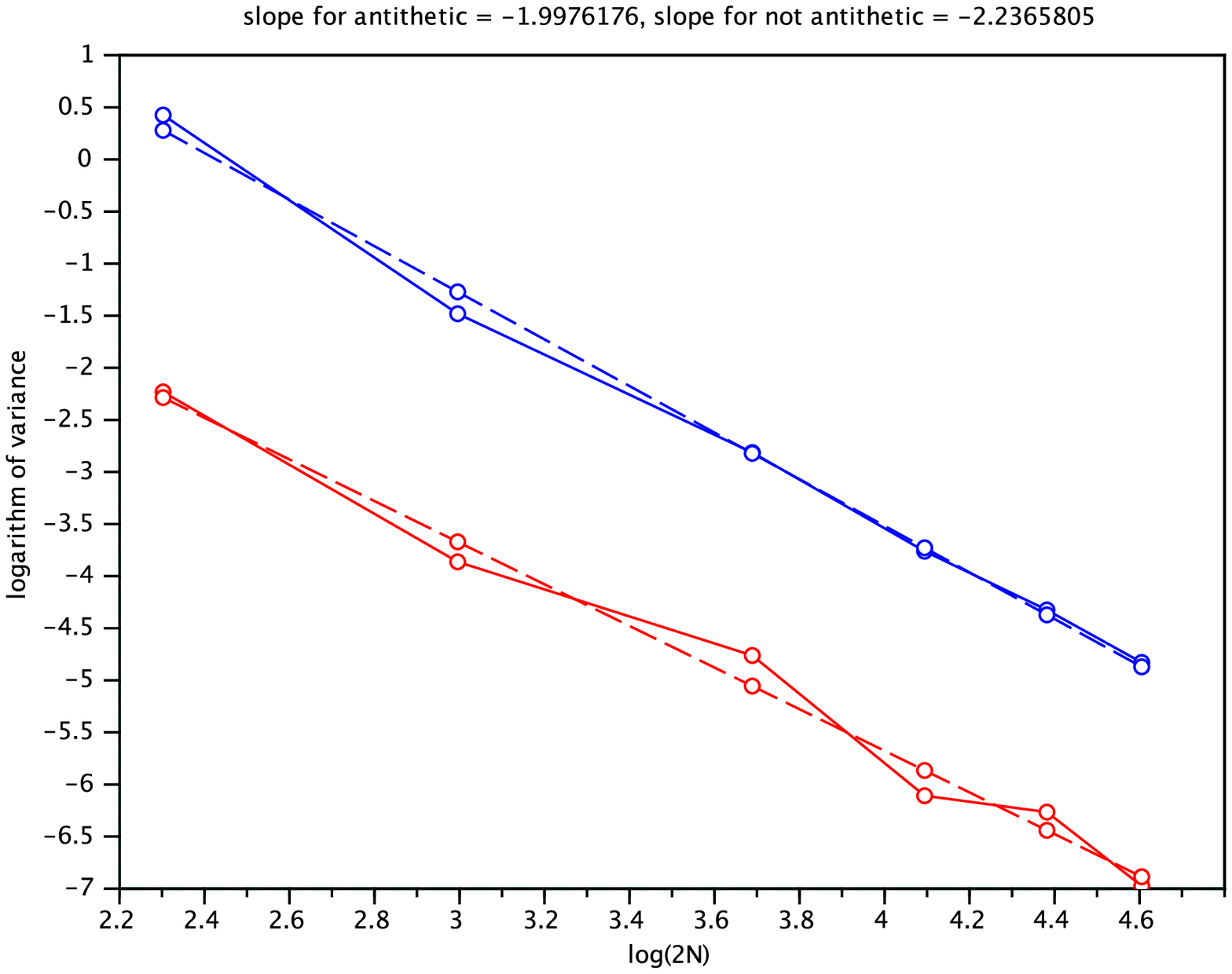}
\quad
\includegraphics[width=5.3cm]{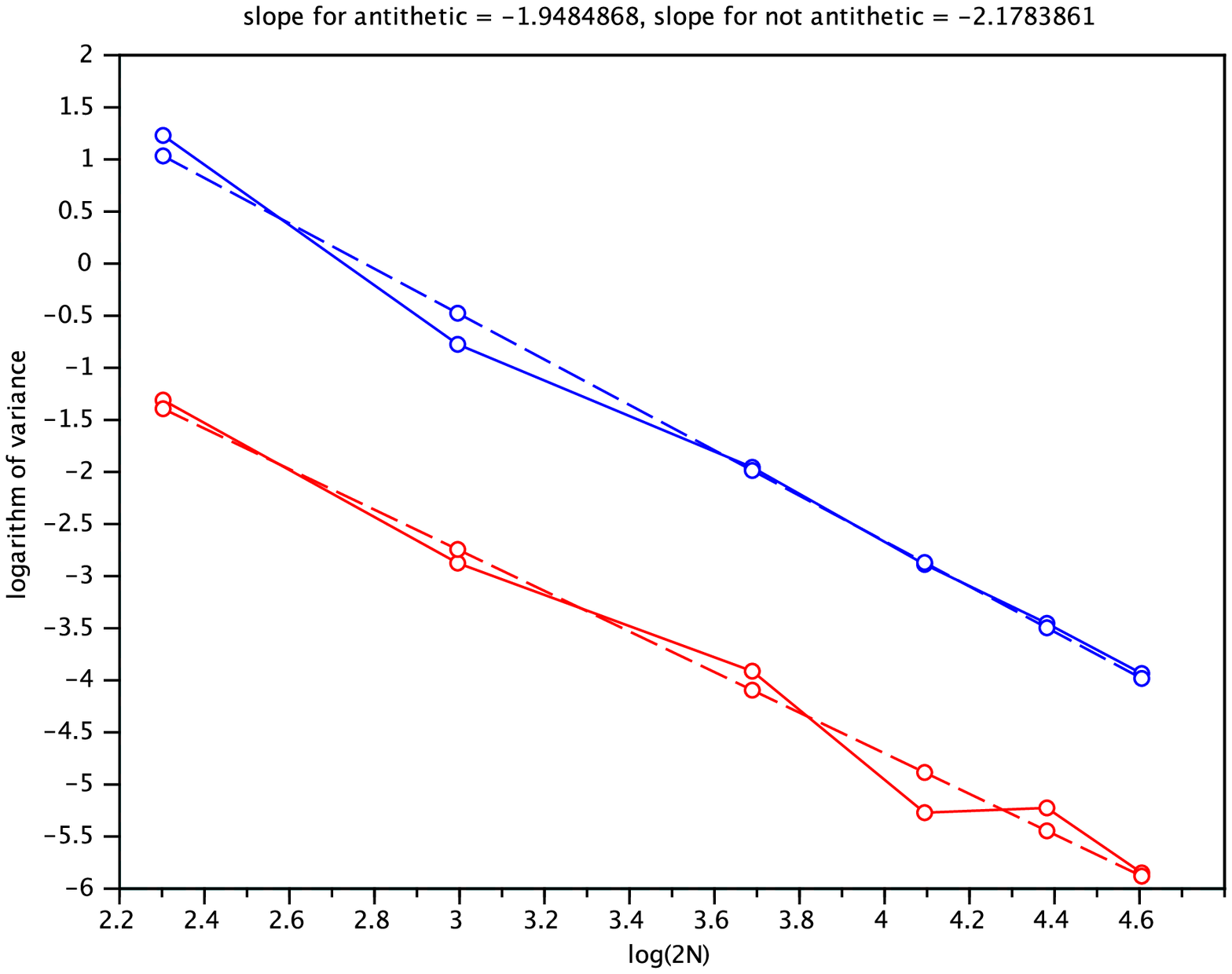}
\includegraphics[width=5.3cm]{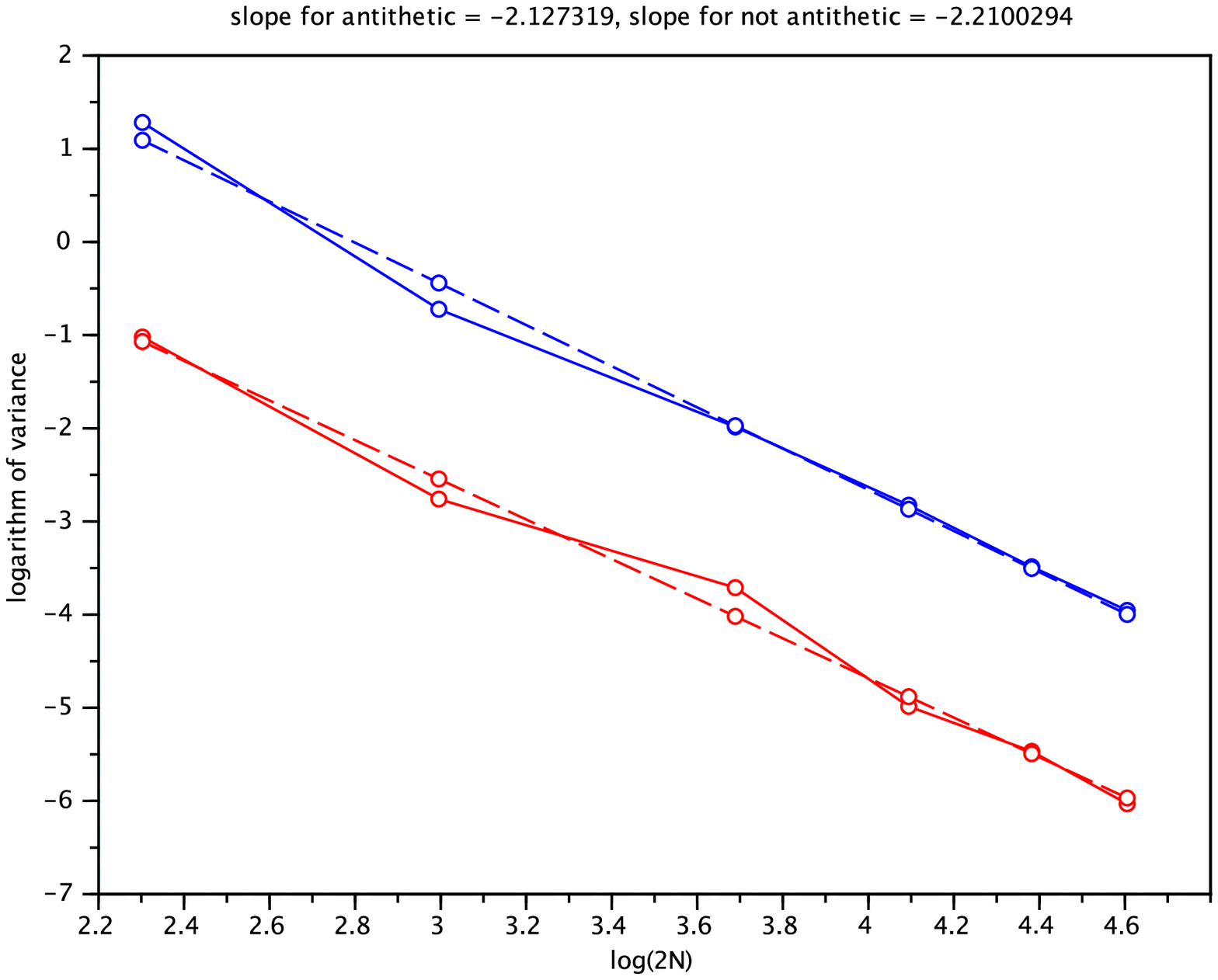}
\quad
\includegraphics[width=5.3cm]{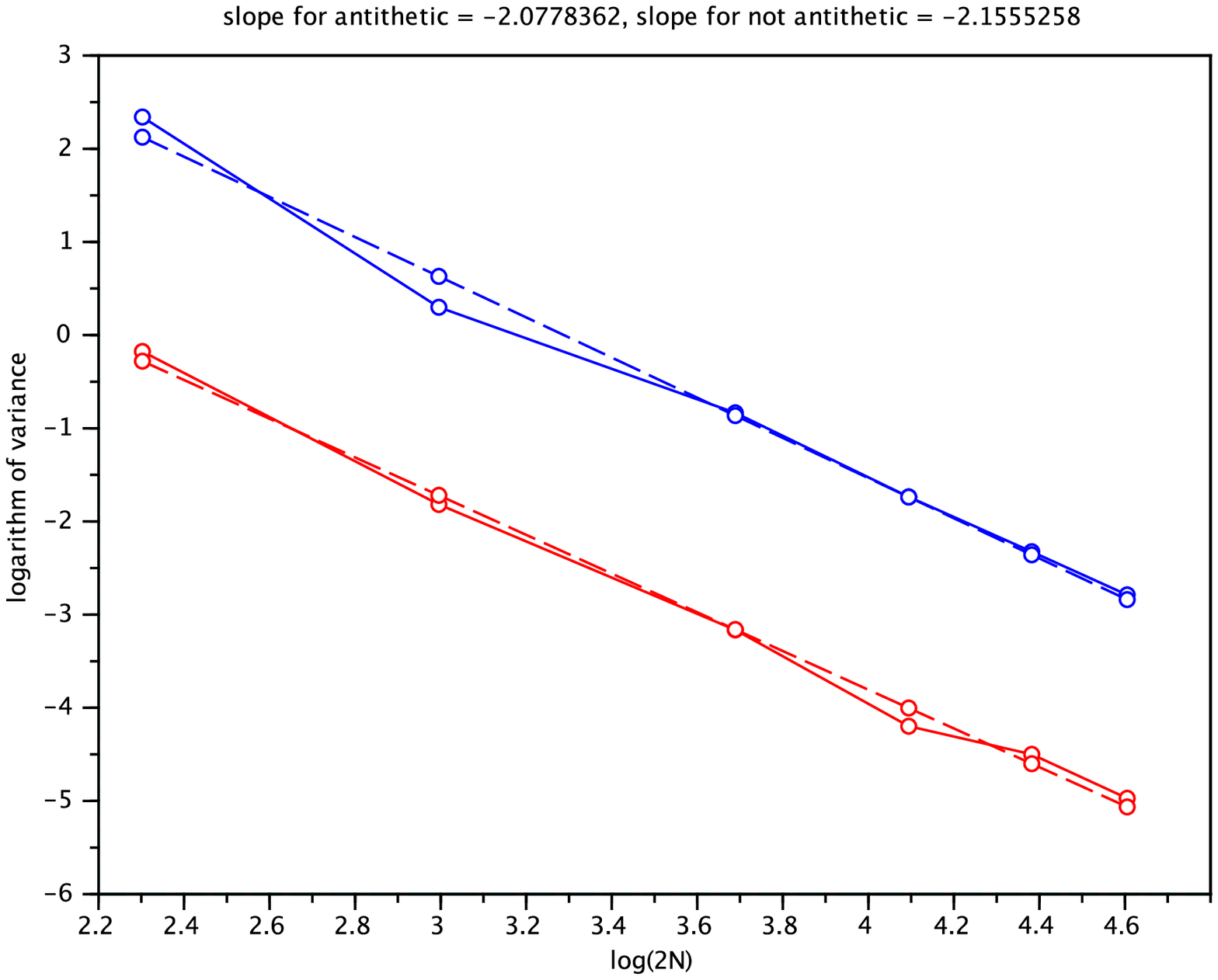}
\includegraphics[width=5.3cm]{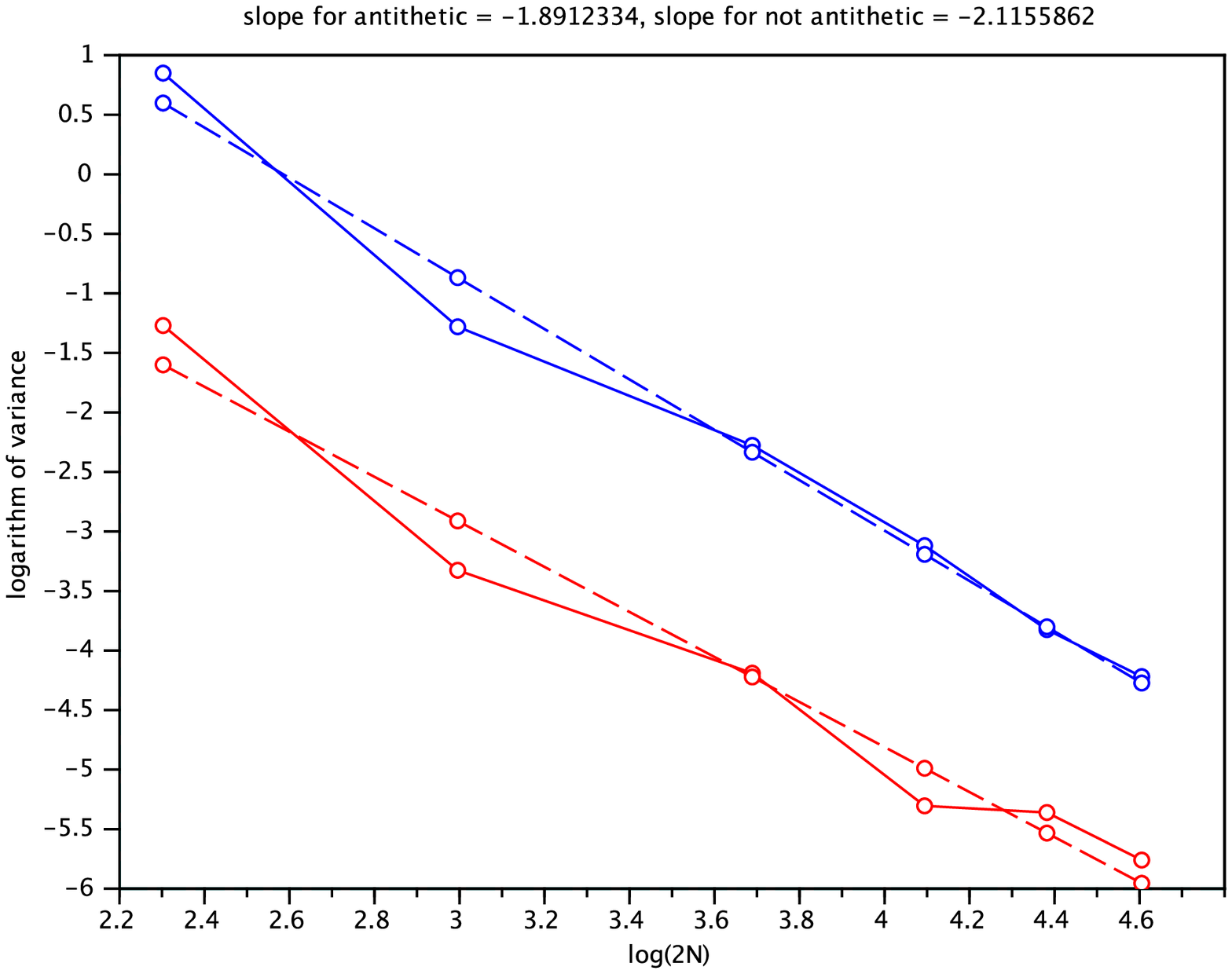}
\quad
\includegraphics[width=5.3cm]{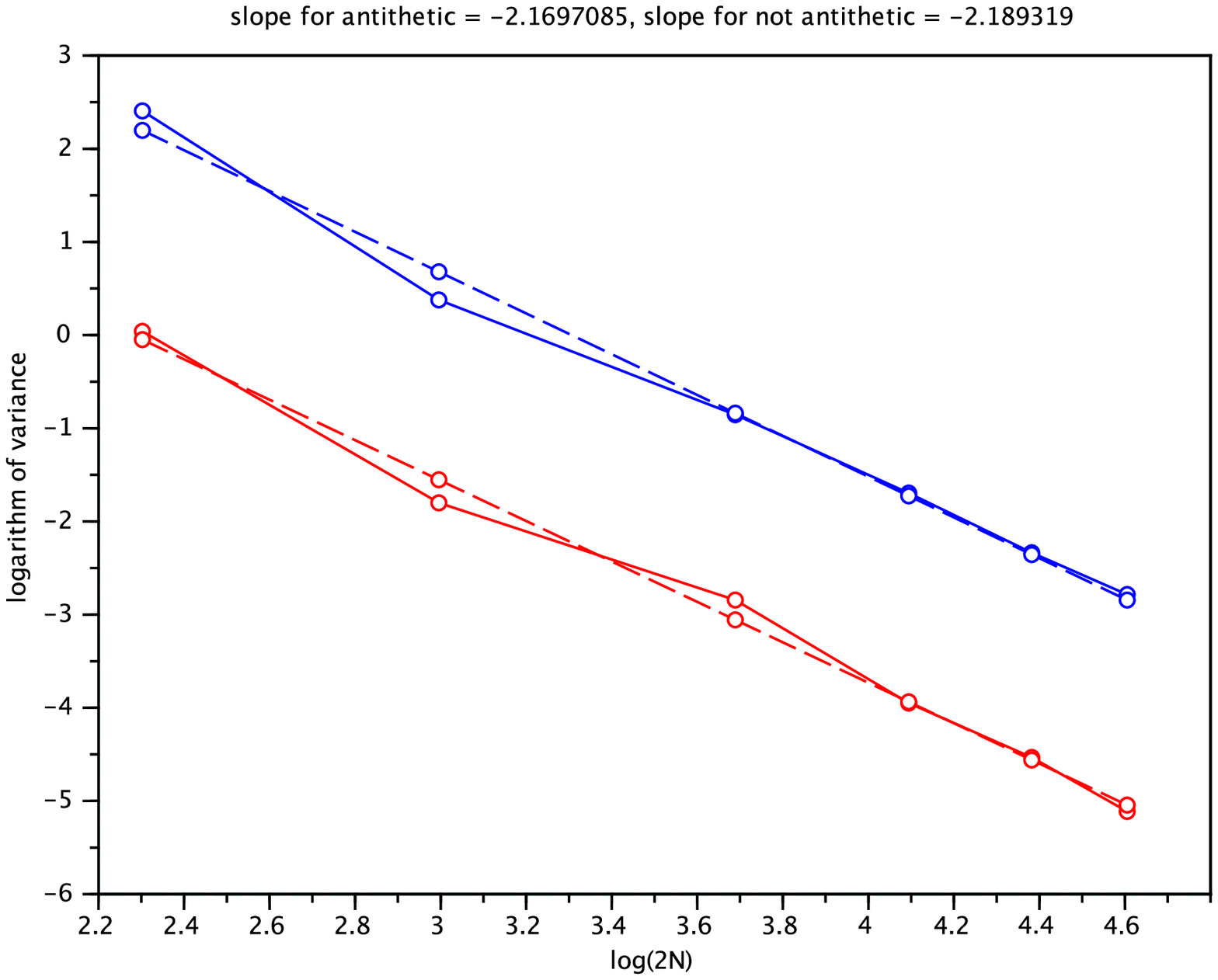}
\includegraphics[width=5.3cm]{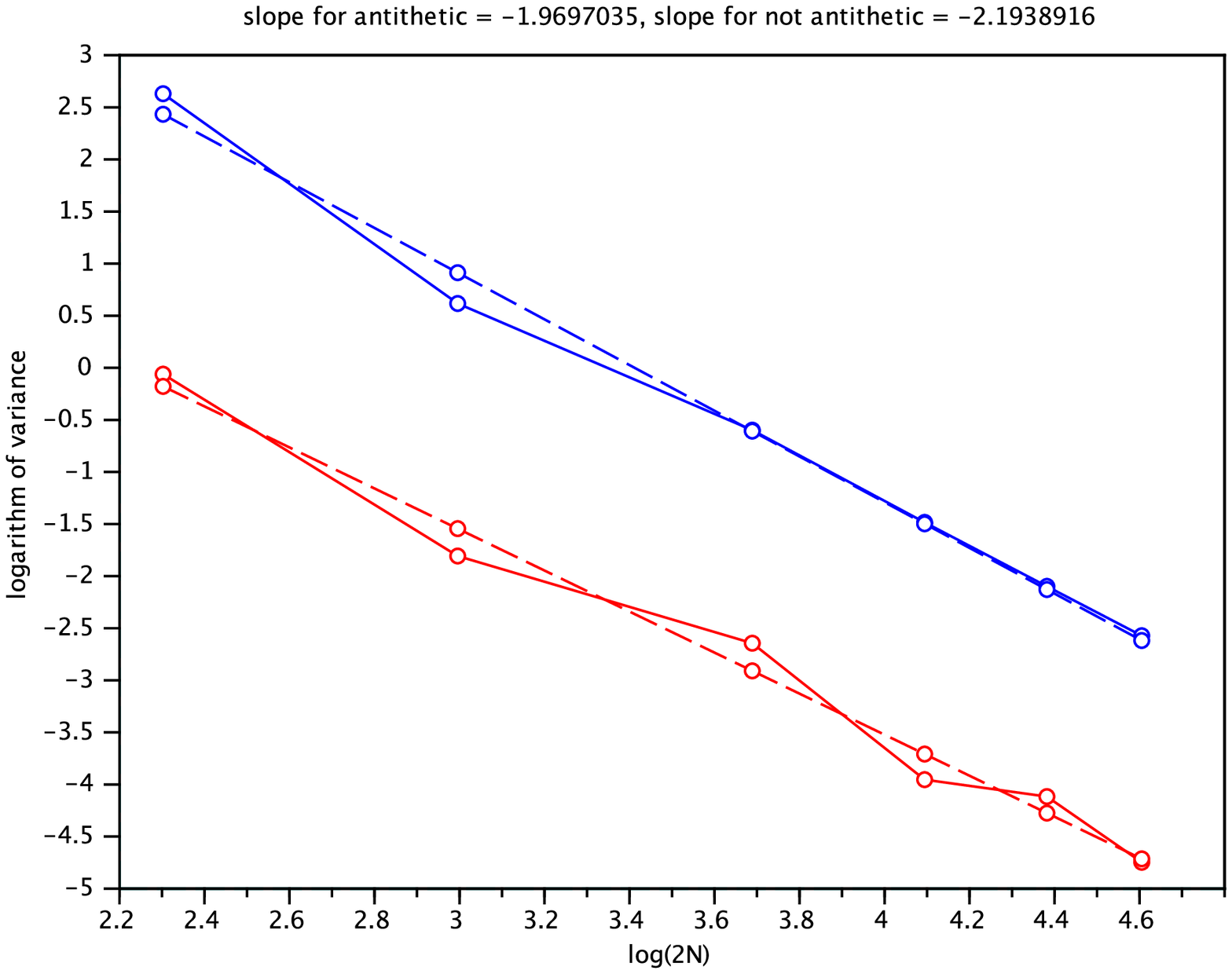}
\quad
\includegraphics[width=5.3cm]{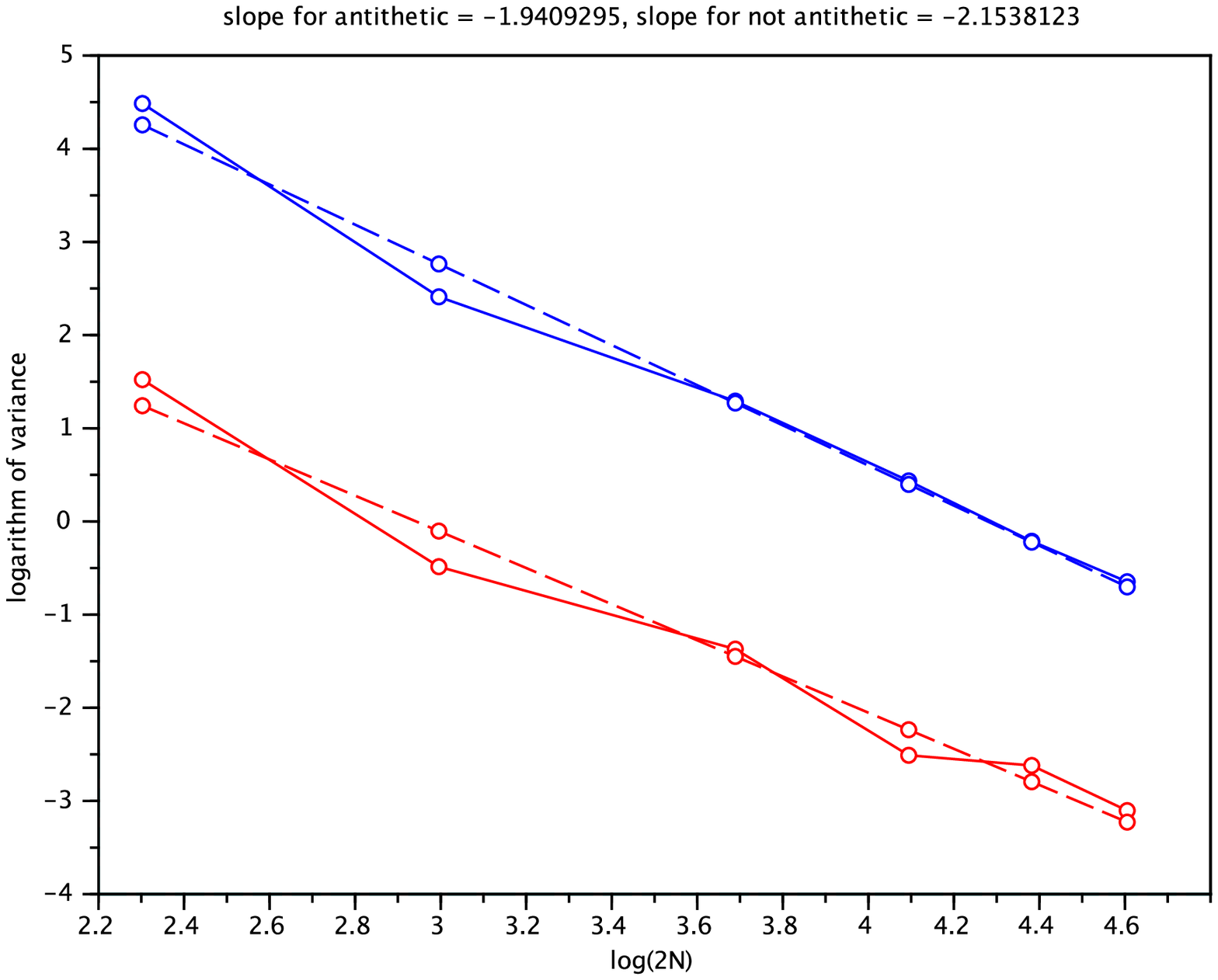}
\caption{Test-Case 3: Variances~\eqref{eq:def_variances} as a function of $N$ (Blue: Monte Carlo approach;
  Red: Antithetic Variable approach). The quantities of interest are
  the same as on Figure~\ref{per-homog-wrt-dof}.
\label{per-non-homog-log-var-wrt-dof}}
\end{center}
\end{figure}

On Table~\ref{tab:non-homog-var-red}, we report the variance reduction
ratios~\eqref{eq:def_R} (with the same convention as in
Table~\ref{tab:p-lap-var-red}). Results are quantitatively similar to
the ones obtained on Table~\ref{tab:non-linear-var-red}: we do observe a
robust variance reduction, even in cases for which theoretical support
is still currently missing. 

\begin{table}[!h]
\begin{center}
\begin{tabular}{|c|c|c|c|c|c|c|c|c|}
$2N$ & $W^\star_N$ & $\partial_{\xi_1} W^\star_N$ & $\partial_{\xi_2}
W^\star_N$ & $\partial_{\xi_1 \xi_1} W^\star_N$ & 
$\partial_{\xi_1 \xi_2} W^\star_N$ & $\partial_{\xi_2 \xi_2} W^\star_N$ &
$\xi \cdot \partial_\xi W^\star_N$ & $\xi^T \partial^2_\xi W^\star_N \xi$ \\
\hline
10 & 14.26 & 12.69 & 10.00 & 12.38 & 8.333 & 10.65 & 14.76 & 19.37
\\
20 & 10.82 & 8.166 & 7.669 & 8.304 & 7.730 & 8.827 & 11.29 & 18.11 
\\
40 & 7.014 & 7.077 & 5.613 & 10.28 & 6.776 & 7.310 & 7.731 & 14.32 
\\
60 & 10.45 & 10.84 & 8.666 & 11.72 & 8.896 & 9.524 & 11.82 & 19.01 
\\
80 & 6.961 & 5.880 & 7.250 & 8.800 & 4.646 & 8.996 & 7.522 & 11.10  
\\
100 & 8.543 & 6.780 & 7.970 & 8.873 & 4.669 & 10.26 & 8.798 & 11.66
\\
\end{tabular}
\caption{Test Case 3: Variance reduction ratios~\eqref{eq:def_R}.
\label{tab:non-homog-var-red}}
\end{center}
\end{table}
    
\section*{Acknowledgments} 

The work of FL and WM is partially supported by ONR under Grant 
N00014-09-1-0470. WM gratefully acknowledges the support from Labex MMCD
(Multi-Scale Modelling \& Experimentation of Materials for Sustainable
Construction) under contract ANR-11-LABX-0022. We also wish to thank
Claude Le Bris for enlightning discussions.

\end{document}